\pdfoutput=1

\documentclass[11pt]{amsart}
\usepackage{amsmath,amsthm, amssymb, amsfonts, amscd}
\usepackage[mathscr]{eucal}
\usepackage{epsfig}
\usepackage{graphicx}
\usepackage{psfrag}
\usepackage[usenames,dvipsnames]{color}
\usepackage{xy}
\usepackage{subfigure}

\newtheorem{mytheorem}{Theorem}
\newtheorem{theorem}{Theorem}[section]
\newtheorem{mycorollary}[mytheorem]{Corollary}
\newtheorem{corollary}[theorem]{Corollary}
\newtheorem{lemma}[theorem]{Lemma}
\newtheorem{proposition}[theorem]{Proposition}

\newtheorem{conjecture}[theorem]{Conjecture}
\newtheorem{myconjecture}[mytheorem]{Conjecture}
\newtheorem*{hodgetheorem}{Hodge Decomposition Theorem}
\newtheorem*{hmftheorem}{Hodge--Morrey--Friedrichs Decomposition Theorem}
\newtheorem*{greenformula}{Green's Formula}

\newtheorem*{theoremcpn}{\thmref{thm:cpn}}
\newtheorem*{theoremg2rn}{\thmref{thm:g2rn}}
\newtheorem*{theoremeigenvalues}{\thmref{thm:pdangleseigenvalues}}

\newtheorem*{theoremmixedboundary}{\thmref{thm:mixedcupproductboundary}}

\newtheorem*{corollaryeuclidean}{Corollary~\ref{cor:mixedeuclidean}}
\newtheorem*{kertheorem}{\thmref{thm:kernel}}

\theoremstyle{definition}

\theoremstyle{definition}

\newcommand{\thmref}[1]{Theorem~\ref{#1}}
\newcommand{\propref}[1]{Proposition~\ref{#1}}

\newcommand{\lemref}[1]{Lemma~\ref{#1}}

\newcommand{\hp}{{\mathcal H}^p}

\newcommand{\hq}{{\mathcal H}^q}
\newcommand{\hpq}{{\mathcal H}^{p+q}}

\input epsf

\newcommand{\mfr}[1]{\mathfrak{#1}}

\newcommand{\CP}{{\mathbb{CP}}}

\newcommand{\til}[1]{{\widetilde{#1}}}
\makeatletter
\def\imod#1{\allowbreak\mkern10mu({\operator@font mod}\,\,#1)}
\makeatother

\input xy
\xyoption{all}

\usepackage[colorlinks=true, urlcolor=MidnightBlue, linkcolor=MidnightBlue, citecolor=MidnightBlue, pdftitle={Poincaré duality angles on Riemannian manifolds with boundary}, pdfauthor={Clayton Shonkwiler}, pdfsubject={Riemannian geometry, inverse problems}, pdfkeywords={Hodge theorem, Dirichlet-to-Neumann map, Hilbert transform}]{hyperref}
\usepackage[nohug]{diagrams}\diagramstyle[labelstyle=\scriptstyle]

\begin{document}

\title[Poincar\'e duality angles]{Poincar\'e duality angles for Riemannian manifolds with boundary}  

\author{Clayton Shonkwiler}
\address{Department of Mathematics \\ Haverford College}
\email{cshonkwi@haverford.edu}
\urladdr{\href{http://www.haverford.edu/math/cshonkwi/}{http://www.haverford.edu/math/cshonkwi/}}

\date{\today}
\maketitle

\begin{abstract}

On a compact Riemannian manifold with boundary, the absolute and relative cohomology groups appear as certain subspaces of harmonic forms.  DeTurck and Gluck showed that these concrete realizations of the cohomology groups decompose into orthogonal subspaces corresponding to cohomology coming from the interior and boundary of the manifold.  The principal angles between these interior subspaces are all acute and are called {\it Poincar\'e duality angles}.  This paper determines the Poincar\'e duality angles of a collection of interesting manifolds with boundary derived from complex projective spaces and from Grassmannians, providing evidence that the Poincar\'e duality angles measure, in some sense, how ``close'' a manifold is to being closed.

This paper also elucidates a connection between the Poincar\'e duality angles and the Dirichlet-to-Neumann operator for differential forms, which generalizes the classical Dirichlet-to-Neumann map arising in the problem of Electrical Impedance Tomography.  Specifically, the Poincar\'e duality angles are essentially the eigenvalues of a related operator, the Hilbert transform for differential forms.  This connection is then exploited to partially resolve a question of Belishev and Sharafutdinov about whether the Dirichlet-to-Neumann map determines the cup product structure on a manifold with boundary.

\end{abstract}

\section{Introduction} 
\label{chap:introduction}

Consider a closed, smooth, oriented Riemannian manifold $M^n$.  For any $p$ with $0 \leq p \leq n$, the Hodge Decomposition Theorem \cite{Hodge34, Kodaira} says that the $p$th cohomology group $H^p(M; \mathbb{R})$ is isomorphic to the space of closed and co-closed differential $p$-forms on $M$.  Thus, the space of such forms (called {\it harmonic fields} by Kodaira) is a concrete realization of the cohomology group $H^p(M; \mathbb{R})$ inside the space $\Omega^p(M)$ of all $p$-forms on $M$.

Since $M$ is closed, $\partial M = \emptyset$, so $H^p(M; \mathbb{R}) = H^p(M, \partial M; \mathbb{R})$ and thus the concrete realizations of $H^p(M; \mathbb{R})$ and $H^p(M, \partial M; \mathbb{R})$ coincide.  This turns out not to be true for manifolds with non-empty boundary.

When $M^n$ is a compact, smooth, oriented Riemannian manifold with non-empty boundary $\partial M$, the relevant version of the Hodge Decomposition Theorem was proved by Morrey \cite{Morrey} and Friedrichs \cite{Friedrichs}.  It gives concrete realizations of both $H^p(M; \mathbb{R})$ and $H^p(M, \partial M; \mathbb{R})$ inside the space of harmonic $p$-fields on $M$.  Not only do these spaces not coincide, they intersect only at 0.  Thus, a somewhat idiosyncratic way of distinguishing the closed manifolds from among all compact Riemannian manifolds is as those manifolds for which the concrete realizations of the absolute and relative cohomology groups coincide.  It seems plausible that the relative positions of the concrete realizations of $H^p(M; \mathbb{R})$ and $H^p(M, \partial M; \mathbb{R})$ might, in some ill-defined sense, determine how close the manifold $M$ is to being closed.

The relative positions of these spaces was described more completely by DeTurck and Gluck \cite{DG}.  They noted that $H^p(M; \mathbb{R})$ and $H^p(M, \partial M; \mathbb{R})$ each have one portion consisting of those cohomology classes coming from the boundary of $M$ and another portion consisting of those classes coming from the ``interior'' of $M$.  DeTurck and Gluck showed that these interior and boundary portions manifest themselves as orthogonal subspaces inside the concrete realizations of $H^p(M; \mathbb{R})$ and $H^p(M, \partial M; \mathbb{R})$.  This leads to a refinement of the Hodge--Morrey--Friedrichs decomposition which says, in part, that
\begin{enumerate}
	\item the concrete realizations of $H^p(M; \mathbb{R})$ and $H^p(M, \partial M; \mathbb{R})$ meet only at the origin,
	\item the boundary subspace of each is orthogonal to all of the other, and
	\item the principal angles between the interior subspaces are all acute.
\end{enumerate}
This behavior is depicted in Figure~\ref{fig:2planes1}.

\begin{figure}[htbp]
	\centering
		\includegraphics[scale=1]{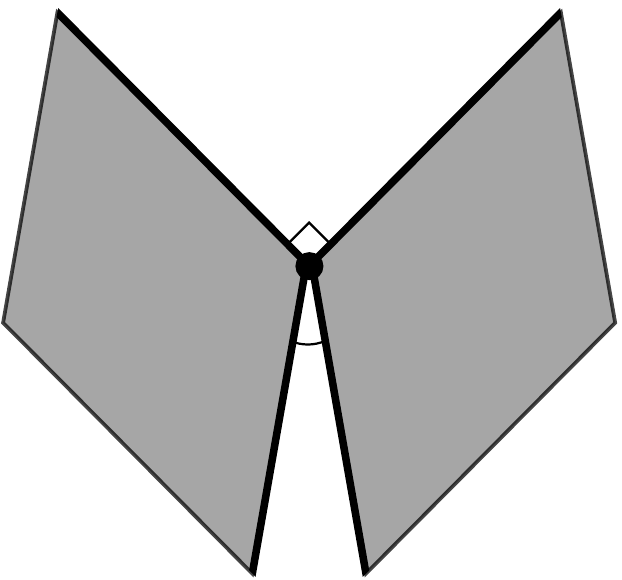}
		\put(-155,95){\rotatebox{-45}{{$H^p(M; \mathbb{R})$}}}
		\put(-144,133){\rotatebox{-45}{\small{\sc{boundary}}}}
		\put(-76,98){\rotatebox{45}{\small{\sc{boundary}}}}
		\put(-72,59){\rotatebox{45}{{$H^p(M, \partial M; \mathbb{R})$}}}
		\put(-110,30){\rotatebox{80}{\small{\sc{interior}}}}
		\put(-82,67){\rotatebox{-80}{\small{\sc{interior}}}}
	\caption{The concrete realizations of the absolute and relative cohomology groups}
	\label{fig:2planes1}
\end{figure}
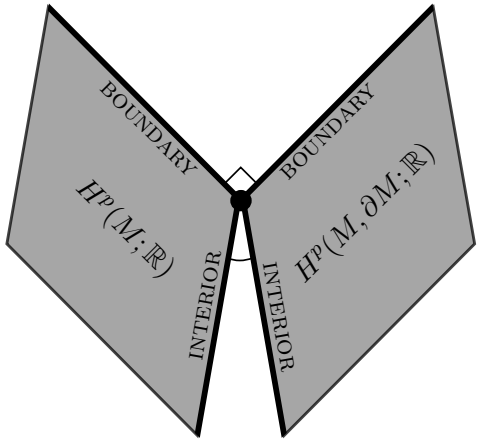

The principal angles between the interior subspaces of the concrete realizations of $H^p(M; \mathbb{R})$ and $H^p(M, \partial M; \mathbb{R})$ are invariants of the Riemannian manifold with boundary $M$ and were christened {\it Poincar\'e duality angles} by DeTurck and Gluck. Since all of the cohomology of a closed manifold is interior and since the concrete realizations of the absolute and relative cohomology groups coincide on such a manifold, it seems reasonable to guess that the Poincar\'e duality angles go to zero as a manifold closed up.

One aim of this paper is to provide evidence for this hypothesis by determining the Poincar\'e duality angles for some interesting manifolds which are intuitively close to being closed.  

For example, consider the complex projective space $\CP^n$ with its usual Fubini-Study metric and define the manifold
\[
	M_r := \CP^n - B_r(x)
\]
obtained by removing a ball of radius $r$ centered at the point $x \in \CP^n$.  If the Poincar\'e duality angles do measure how close a manifold is to being closed, the Poincar\'e duality angles of $M_r$ should be small when $r$ is near zero.

\begin{mytheorem}\label{thm:cpn}
	For $1 \leq k \leq n-1$ there is a non-trivial Poincar\'e duality angle $\theta_r^{2k}$ between the concrete realizations of $H^{2k}(M_r; \mathbb{R})$ and $H^{2k}(M_r, \partial M_r ; \mathbb{R})$ which is given by
	\begin{equation*}
		\cos \theta_r^{2k} = \frac{1-\sin^{2n}r}{\sqrt{(1+\sin^{2n}r)^2 + \frac{(n-2k)^2}{k(n-k)}\sin^{2n}r}}.
	\end{equation*}
\end{mytheorem}
Indeed, as $r\to 0$, the Poincar\'e duality angles $\theta_r^{2k} \to 0$.  Moreover, as $r$ approaches its maximum value of $\pi/2$, the $\theta_r^{2k} \to \pi/2$.

\thmref{thm:cpn} immediately generalizes to other non-trivial $D^2$-bundles over $\CP^{n-1}$ with an appropriate metric, so this asymptotic behavior of the Poincar\'e duality angles is not dependent on being able to cap off the manifold with a ball.

Removing a ball around some point in a closed manifold is just a special case of removing a tubular neighborhood of a submanifold.  With this in mind, consider $G_2 \mathbb{R}^{n+2}$, the Grassmannian of oriented 2-planes in $\mathbb{R}^{n+2}$, and define
\[
	N_r := G_2 \mathbb{R}^{n+2} - \nu_r \left(G_1 \mathbb{R}^{n+1}\right),
\]
where $\nu_r\left(G_1 \mathbb{R}^{n+1} \right)$ is the tubular neighborhood of radius $r$ around the subGrassmannian $G_1 \mathbb{R}^{n+1}$.

\begin{mytheorem}\label{thm:g2rn}
	For $1 \leq k \leq n-1$ there is exactly one Poincar\'e duality angle $\theta_r^{2k}$ between the concrete realizations of $H^{2k}(N_r; \mathbb{R})$ and $H^{2k}(N_r, \partial N_r; \mathbb{R})$ given by
	\[
		\cos \theta_r^{2k} = \frac{1-\sin^nr}{\sqrt{(1+\sin^n r)^2 + \frac{(n-2k)^2}{k(n-k)}\sin^nr}}.
	\]
\end{mytheorem}

Again, the $\theta_r^{2k} \to 0$ as $r \to 0$ and $\theta_r^{2k} \to \pi/2$ as $r$ approaches its maximum value of $\pi/2$.

When $n=2$, the Grassmannian $G_2 \mathbb{R}^4$ is isometric to  $S^2 \left(1/\sqrt{2}\right) \times S^2 \left(1/\sqrt{2}\right)$ and the subGrassmannian $G_1 \mathbb{R}^3$ corresponds to the anti-diagonal 2-sphere $\til{\Delta}$ (cf.\ \cite{GluckWarner}).  Hence,  the Riemannian manifold with boundary $$S^2\left(1/\sqrt{2}\right) \times S^2\left(1/\sqrt{2}\right) - \nu_r\left(\til{\Delta}\right)$$ has a single Poincar\'e duality angle $\theta_r^2$ in dimension 2 given by
\[
	\cos \theta_r^2 = \frac{1-\sin^2 r}{1+\sin^2 r}.
\]

Theorems~\ref{thm:cpn} and \ref{thm:g2rn} suggest the following conjecture:

\begin{myconjecture}\label{conj:tozero}
	Let $M^m$ be a closed, smooth, oriented Riemannian manifold and let $N^n$ be a closed submanifold of codimension $m-n \geq 2$.  Define the compact Riemannian manifold
	\[
		M_r := M - \nu_r(N),
	\]
	where $\nu_r(N)$ is the open tubular neighborhood of radius $r$ about $N$ (restricting $r$ to be small enough that $\partial M_r$ is smooth).  Then, if $\theta_r^k$ is a Poincar\'e duality angle of $M_r$ in dimension $k$,
	\[
		\theta_r^k = O(r^{m-n})
	\]
	for $r$ near zero.
\end{myconjecture}

The Poincar\'e duality angles seem to be interesting invariants even in isolation, but the other aim of this paper is to show that they are related to the Dirichlet-to-Neumann operator for differential forms, which arises in certain inverse problems of independent interest.

The Dirichlet-to-Neumann operator for differential forms was defined by Joshi and Lionheart \cite{Joshi} and Belishev and Sharafutdinov \cite{Belishev} and generalizes the classical Dirichlet-to-Neumann map for functions which arises in the problem of Electrical Impedance Tomography (EIT).  The EIT problem was first posed by Calder\'on \cite{Calderon} in the context of geoprospecting, but is also of considerable interest in medical imaging.

The Dirichlet-to-Neumann map $\Lambda$ is a map $\Omega^p(\partial M) \to \Omega^{n-p-1}(\partial M)$ and can be used to define the Hilbert transform $T = d\Lambda^{-1}$.  The connection to the Poincar\'e duality angles is given by the following theorem:

\begin{mytheorem}\label{thm:pdangleseigenvalues}
	If $\theta_1^p, \ldots , \theta_\ell^p$ are the principal angles between the interior subspaces of the concrete realizations of $H^p(M; \mathbb{R})$ and $H^p(M, \partial M; \mathbb{R})$ (i.e. the Poincar\'e duality angles in dimension $p$), then the quantities
	\[
		(-1)^{pn+p+n}\cos^2\theta_i^p
	\]
	are the non-zero eigenvalues of a suitable restriction of $T^2$. 
\end{mytheorem}

In fact, the eigenspaces of the operator $T^2$ determine a direct-sum decomposition of the traces (i.e. pullbacks to the boundary) of harmonic fields on $M$, which in turn leads to the following refinement of a theorem of Belishev and Sharafutdinov:

\begin{mytheorem}\label{thm:kernel}
	Let $\mathcal{E}^p(\partial M)$ be the space of exact $p$-forms on $\partial M$.  Then the dimension of the quotient $\ker \Lambda / \mathcal{E}^p(\partial M)$ is equal to the dimension of the boundary subspace of $H^p(M; \mathbb{R})$.
\end{mytheorem}

Belishev and Sharafutdinov showed that the cohomology groups of $M$ can be completely determined from the boundary data $(\partial M, \Lambda)$ and, in fact, that this data determines the long exact sequence of the pair $(M, \partial M)$.  \thmref{thm:pdangleseigenvalues} shows that the data $(\partial M, \Lambda)$ not only determines the interior and boundary subspaces of the cohomology groups, but detects their relative positions as subspaces of differential forms on $M$.

At the end of their paper, Belishev and Sharafutdinov posed the following question:

\begin{quote}
\emph{Can the multiplicative structure of cohomologies be recovered from our data $(\partial M, \Lambda)$? Till now, the authors cannot answer the question.} 
\end{quote}

The mixed cup product
\[
	\cup: H^p(M; \mathbb{R}) \times H^q(M, \partial M; \mathbb{R}) \to H^{p+q}(M, \partial M; \mathbb{R})
\]
can be reconstructed from the Dirichlet-to-Neumann map when the relative class comes from the boundary subspace, giving a partial answer to Belishev and Sharafutdinov's question:

\begin{mytheorem}\label{thm:mixedcupproductboundary}
	The boundary data $(\partial M, \Lambda)$ completely determines the mixed cup product when the relative cohomology class is restricted to come from the boundary subspace.  
\end{mytheorem}

When the manifold $M$ occurs as a region in Euclidean space, all relative cohomology classes come from the boundary subspace, so \thmref{thm:mixedcupproductboundary} has the following immediate corollary:

\begin{mycorollary}\label{cor:mixedeuclidean}
	If $M^n$ is a compact region in $\mathbb{R}^n$, the boundary data $(\partial M, \Lambda)$ completely determines the mixed cup product on $M$.
\end{mycorollary}

The expression for the reconstruction of the mixed cup product given in the proof of \thmref{thm:mixedcupproductboundary} makes sense even when the relative class does not come from the boundary subspace, suggesting that the data $(\partial M, \Lambda)$ may determine the full mixed cup product.  It remains an interesting question whether the Dirichlet-to-Neumann data determines the absolute or relative cup products on $M$.

\noindent \begin{center} \rule{1in}{.5px}\end{center}
%

\subsection*{Organizational scheme} 
\label{sub:organization}

Section~\ref{chap:preliminaries} gives the necessary background on Poincar\'e duality angles and the Dirichlet-to-Neumann map, including a complete proof of the existence of the Poincar\'e duality angles.  Sections~\ref{chap:pdangles_cpn} and \ref{chap:pdangles_g2rn} give proofs of Theorems~\ref{thm:cpn} and \ref{thm:g2rn}, respectively.  The connection between Poincar\'e duality angles and the Dirichlet-to-Neumann map is given in Section~\ref{chap:dnmap}, including the proofs of Theorems~\ref{thm:pdangleseigenvalues}--\ref{thm:mixedcupproductboundary}.  The statements of many of the theorems in this introduction are intentionally somewhat vague so as not to overwhelm the basic story with technical details; precise restatements of the theorems are given in the appropriate sections.


\subsection*{Acknowledgements} 
\label{sub:acknowledgements}
This paper communicates the results of my Ph.D. thesis at the University of Pennsylvania.  It could hardly exist without the encouragement and support of my advisors, Dennis DeTurck and Herman Gluck, who deserve my deepest gratitude.  I also want to thank David Shea Vela-Vick and Rafal Komendarczyk for their suggestions and willingness to listen.



\section{Preliminaries} 
\label{chap:preliminaries}

\subsection{Poincar\'e duality angles} 
\label{sec:poincar'e_duality_angles}
DeTurck and Gluck's Poincar\'e duality angles arise from a refinement of the classical Hodge--Morrey--Friedrichs decomposition for Riemannian manifolds with boundary.  This section reviews the Hodge decomposition theorem for closed manifolds and the Hodge--Morrey--Friedrichs decomposition for manifolds with boundary in building up to the definition of the Poincar\'e duality angles and the proof of DeTurck and Gluck's main theorem.

\subsubsection{The Hodge decomposition for closed manifolds} 
\label{sub:the_hodge_decomposition_for_closed_manifolds}
Let $M^n$ be a closed, oriented, smooth Riemannian manifold of dimension $n$.  For each $p$ between $0$ and $n$ let $\Omega^p(M)$ be the space of smooth differential $p$-forms on $M$ and let $\Omega(M) = \bigoplus_{i=0}^n \Omega^i(M)$ be the algebra of all differential forms on $M$.

For each $p$, the exterior derivative $d: \Omega^p(M) \to \Omega^{p+1}(M)$ is defined independently of the Riemannian metric.  A differential form $\omega \in \Omega^p(M)$ is closed if $d\omega = 0$ and is exact if $\omega = d\eta$ for some $\eta \in \Omega^{p-1}(M)$.  Letting $\mathcal{C}^p(M)$ denote the space of closed $p$-forms on $M$ and $\mathcal{E}^p(M)$ the space of exact $p$-forms, the $p$th de Rham cohomology group is defined as the quotient
\[
	H^p_{\text{dR}}(M) := \mathcal{C}^p/\mathcal{E}^p.
\]
De Rham's theorem \cite{deRham} states that
\[
	H^p_{\text{dR}}(M) \cong H^p(M; \mathbb{R}),
\]
the $p$th singular cohomology group with real coefficients.  If $\Omega(M)$ is equipped with an inner product, it is natural to expect that the cohomology group $H^p(M; \mathbb{R})$ will be realized as the orthogonal complement of the space of exact $p$-forms inside the space of closed $p$-forms on $M$.  It is the Riemannian metric which gives an inner product on $\Omega(M)$.  

The metric allows the Hodge star
\[
	\star: \Omega^p(M) \to \Omega^{n-p}(M)
\]
to be defined for each $p$.  In turn, the co-differential
\[
	\delta = (-1)^{n(p+1)+1} \star d\, \star: \Omega^{p}(M) \to \Omega^{p-1}
\]
and the Hodge Laplacian
\[
	\Delta = d\delta + \delta d
\]
are defined using the exterior derivative and the Hodge star.

Then the $L^2$ inner product on $\Omega^p(M)$ is defined as 
\[
	\langle \alpha , \beta \rangle_{L^2} := \int_M \alpha \wedge \star\, \beta
\]
for $\alpha, \beta \in \Omega^p(M)$.  This inner product is extended to all of $\Omega(M)$ by declaring that, for $i \neq j$, $\Omega^i(M)$ is orthogonal to $\Omega^j(M)$.  

Consider the following subspaces of $\Omega^p(M)$:
\begin{align*}
	\mathcal{E}^p(M) & := \{\omega \in \Omega^p(M) : \omega = d\eta \text{ for some } \eta \in \Omega^{p-1}(M)\}\\
	c\mathcal{E}^p(M) & := \{\omega \in \Omega^p(M) : \omega = \delta \xi \text{ for some } \xi \in \Omega^{p+1}(M) \} \\
	\hp(M) & := \{\omega \in \Omega^p(M) : d\omega = 0 \text{ and } \delta \omega = 0\},
\end{align*}
the spaces of exact $p$-forms, co-exact $p$-forms and harmonic $p$-fields, respectively.  On closed manifolds the space $\hp(M)$ of harmonic $p$-fields coincides with the space
\[
	\widehat{\mathcal{H}}^p(M) := \{\omega \in \Omega^p(M) : \Delta \omega = 0\}
\]
of harmonic $p$-forms, but this is special to closed manifolds and fails on manifolds with boundary.

With the above notation in place, the Hodge Decomposition Theorem can now be stated:

\begin{hodgetheorem}
	Let $M^n$ be a closed, oriented, smooth Riemannian manifold.  For each integer $p$ such that $0 \leq p \leq n$, the space $\Omega^p(M)$ of smooth $p$-forms on $M$ admits the $L^2$-orthogonal decomposition
	\begin{equation}\label{eqn:hodgedecomp}
		\Omega^p(M) = c\mathcal{E}^p(M) \oplus \hp(M) \oplus \mathcal{E}^p(M).
	\end{equation}
	Moreover, the space $\hp(M)$ of harmonic $p$-fields is finite-dimensional and
	\[
		\hp(M) \cong H^p(M; \mathbb{R}).
	\]
\end{hodgetheorem}
The Hodge Decomposition Theorem has its historical roots in the work of Helmholtz \cite{Helmholtz}; the modern version was developed through the work of Hodge \cite{Hodge34, Hodge41}, Weyl \cite{Weyl} and Kodaira \cite{Kodaira}.  A complete proof is given in Warner's book \cite{Warner}.

The fact that $\hp(M) \cong H^p(M; \mathbb{R})$ follows from the decomposition \eqref{eqn:hodgedecomp}, which implies that the space of closed $p$-forms $\mathcal{C}^p(M) = \hp(M) \oplus \mathcal{E}^p(M)$.  In other words, the space of harmonic $p$-fields $\hp(M)$ is the orthogonal complement of $\mathcal{E}^p(M)$ inside the space of closed $p$-forms.  Then, as expected from the de Rham theorem,
\[
	H^p(M; \mathbb{R}) \cong H^p_{\text{dR}}(M) = \mathcal{C}^p(M)/\mathcal{E}^p(M) \cong \hp(M).
\]


\subsubsection{The Hodge--Morrey--Friedrichs decomposition} 
\label{sub:the_hodge_morrey_friedrichs_decomposition}
Here and throughout the remainder of this paper, let $M^n$ be a compact, oriented, smooth Riemannian manifold with non-empty boundary $\partial M$.  Let $i: \partial M \to M$ be the inclusion map.  The space $\Omega^p(M)$ and the maps $d$, $\star$, $\delta$ and $\Delta$ are defined just as in Section~\ref{sub:the_hodge_decomposition_for_closed_manifolds}.  Let $d_\partial$, $\star_\partial$, $\delta_\partial$ and $\Delta_\partial$ denote the exterior derivative, Hodge star, co-derivative and Laplacian on the closed Riemannian manifold $\partial M$.

One consequence of the Hodge theorem for closed manifolds is that the $p$th real cohomology group $H^p(M; \mathbb{R})$ can be realized as the subspace of harmonic $p$-fields $\hp(M)$ inside the space of $p$-forms.  When the boundary is non-empty the space $\hp(M)$ is infinite-dimensional and so is much too big to represent the cohomology.  Also, on a manifold with boundary there are {\it two} different $p$th cohomology groups: the absolute cohomology group $H^p(M; \mathbb{R})$ and the relative cohomology group $H^p(M, \partial M; \mathbb{R})$.  

That being said, the de Rham theorem stated earlier for closed manifolds holds equally well for the absolute cohomology groups on a manifold with boundary: $H^p(M; \mathbb{R})$ is isomorphic to the quotient of the space of closed $p$-forms by the space of exact $p$-forms.  Thus, the concrete realization of the absolute cohomology group $H^p(M; \mathbb{R})$ ought to be the orthogonal complement of $\mathcal{E}^p(M)$ inside the space of closed $p$-forms.

For the relative version, consider those smooth $p$-forms $\omega$ whose pullback $i^*\omega$ to $\partial M$ is zero.  Define
\[
	\Omega^p(M, \partial M) := \{\omega \in \Omega^p(M) : i^*\omega = 0\},
\]
the space of {\it relative $p$-forms} on $M$.  Let 
\begin{align*}
	\mathcal{C}^p(M, \partial M) & := \{\omega \in \Omega^p(M, \partial M) : d\omega = 0\}\\
	\mathcal{E}^p(M, \partial M) & := \{\omega \in \Omega^p(M, \partial M) : \omega = d\eta \text{ for some } \eta \in \Omega^{p-1}(M, \partial M)\}
\end{align*}
denote the subspaces of closed relative $p$-forms and {\it relatively exact $p$-forms}, respectively.  The {\it relative de Rham cohomology groups} can then be defined as
\[
	H^p_{\text{dR}}(M, \partial M) := \mathcal{C}^p(M, \partial M) / \mathcal{E}^p(M, \partial M).
\]

Duff \cite{Duff} proved the relative version of de Rham's theorem, namely that
\[
	H^p(M, \partial M; \mathbb{R}) \cong H^p_{\text{dR}}(M, \partial M).
\]
Thus, the concrete realization of the relative cohomology group $H^p(M, \partial M; \mathbb{R})$ ought to be the orthogonal complement of the subspace of relatively exact $p$-forms inside the space of closed relative $p$-forms.

The definition of the relative $p$-forms introduces the boundary condition $i^*\omega = 0$, which is called the {\it Dirichlet boundary condition}.  A form $\omega$ satisfying the Dirichlet boundary condition can be thought of as ``normal'' to the boundary, since, for any $x \in \partial M$ and $V_1, \ldots , V_p \in T_xM$,
\[
	\omega(V_1, \ldots , V_p) \neq 0
\]
only if one of the $V_i$ has a non-trivial component in the direction of the inward-pointing unit normal vector $\mathcal{N}$.  Forms satisfying the Dirichlet boundary condition are just the relative forms and the relatively exact forms are just those which are exact with a primitive satisfying the Dirichlet boundary condition.

For $x \in \partial M$, let $\pi: T_x M \to T_x \partial M$ be the orthogonal projection.  Then it is natural to think of a form $\omega \in \Omega^p(M)$ as ``tangent'' to the boundary if
\[
	\omega(V_1, \ldots , V_p) = \omega(\pi V_1, \ldots , \pi V_p)
\]
for any $V_1, \ldots , V_p \in T_x M$.  Equivalently, a form $\omega$ is tangent to the boundary if the contraction $\iota_{\mathcal{N}} \omega$ of the unit inward-pointing normal vector $\mathcal{N}$ into $\omega$ is zero.  Therefore, if $\omega$ is tangent to the boundary,
\[
	0 = \star_\partial\, \iota_{\mathcal{N}} \omega = i^*\!\star \omega.
\]
The  condition $i^*\!\star \omega = 0$ is called the {\it Neumann boundary condition}.  If $\omega \in \Omega^p(M)$ satisfies the Neumann condition, then $\star \,\omega \in \Omega^{n-p}(M)$ satisfies the Dirichlet condition.  Likewise, $\star$~maps Dirichlet forms to Neumann forms.

Let $c\mathcal{E}^p(M)$, $\hp(M)$ and $\mathcal{E}^p(M)$ be the spaces of co-exact $p$-forms, harmonic $p$-fields and exact $p$-forms, respectively.  Throughout what follows, juxtaposition of letters denotes intersections; for example, 
\[
	c\mathcal{E}\hp(M) := c\mathcal{E}^p(M) \cap \hp(M) \quad \text{and} \quad \mathcal{E}\hp(M) := \mathcal{E}^p(M) \cap \hp(M) .
\]
Unlike in the case of a closed manifold, these spaces are non-trivial on a compact manifold with boundary.

Subscripts $N$ and $D$ indicate Neumann and Dirichlet boundary conditions, respectively:
\begin{align*}
	c\mathcal{E}^p_N(M) & := \{\omega \in \Omega^p(M) : \omega = \delta \xi \text{ for some } \xi \in \Omega^{p+1}(M) \text{ where } i^*\!\star \xi = 0\} \\
	\hp_N(M) & := \{\omega \in \Omega^p(M) : d\omega = 0, \delta \omega = 0, i^*\!\star \omega = 0\} \\
	\hp_D(M) & := \{\omega \in \Omega^p(M) : d\omega = 0, \delta \omega = 0, i^*\omega = 0\} \\
	\mathcal{E}^p_D(M) & := \{\omega \in \Omega^p(M) : \omega = d\eta \text{ for some } \eta \in \Omega^{p-1}(M) \text{ where } i^*\eta = 0\}.
\end{align*}
A key subtlety is that the boundary conditions apply to the {\it primitive} of $\omega$ in the definitions of $c\mathcal{E}^p_N(M)$ and $\mathcal{E}^p_D(M)$, whereas they apply to the form $\omega$ itself in the definitions of $\hp_N(M)$ and $\hp_D(M)$.  

If $\omega \in \mathcal{E}^p_D(M)$, then $\omega = d\eta$ for some Dirichlet form $\eta \in \Omega^{p-1}(M)$ and
\[
	i^*d\eta = d_\partial i^*\eta = 0,
\]
so elements of $\mathcal{E}^p_D(M)$ do themselves satisfy the Dirichlet boundary condition.  In particular, this means that $\mathcal{E}^p_D(M)$ is precisely the space of relatively exact $p$-forms.  Likewise, if $\omega \in c\mathcal{E}_N^p(M)$, then $\omega = \delta \xi$ for some Neumann form $\xi \in \Omega^{p+1}(M)$ and
\[
	i^*\!\star \delta \xi = (-1)^{p + 1} i^* d\star \xi = (-1)^{p+1} d_\partial\, i^*\! \star \xi = 0,
\]
so forms in $c\mathcal{E}^p_N(M)$ satisfy the Neumann boundary condition.  

On a compact Riemannian manifold with boundary, the analogue of the Hodge theorem is the following, which combines the work of Morrey \cite{Morrey} and Friedrichs~\cite{Friedrichs}:

\begin{hmftheorem}
	Let $M$ be a compact, oriented, smooth Riemannian manifold with non-empty boundary $\partial M$.  Then the space $\Omega^p(M)$ can be decomposed as
	\begin{align}
\label{eqn:hmf1}		\Omega^p(M) & = c\mathcal{E}^p_N(M)  \oplus \hp_N(M) \oplus \mathcal{E}\hp(M) \oplus \mathcal{E}^p_D(M) \\
\label{eqn:hmf2}		& = c\mathcal{E}^p_N(M) \oplus c\mathcal{E}\hp(M)  \oplus  \hp_D(M) \oplus \mathcal{E}^p_D(M),
	\end{align}
	where the direct sums are $L^2$-orthogonal.  Moreover, 
	\begin{align*}
		H^p(M; \mathbb{R}) & \cong \hp_N(M) \\
		H^p(M, \partial M; \mathbb{R}) & \cong \hp_D(M)
	\end{align*}
\end{hmftheorem}

Morrey proved that
\begin{equation}\label{eqn:morrey}
	\Omega^p(M) = c\mathcal{E}^p_N(M) \oplus \hp(M) \oplus \mathcal{E}^p_D(M)
\end{equation}
and Friedrichs  gave the two decompositions of the harmonic fields:
\begin{align}
	\label{eqn:friedrichs} \hp(M) & = \hp_N(M) \oplus \mathcal{E}\hp(M)  \\
	\label{eqn:friedrichs2} & = c\mathcal{E}\hp(M) \oplus \hp_D(M);
\end{align}
both were influenced by the work of Duff and Spencer \cite{DuffSpencer}.  The orthogonality of the components follows immediately from Green's Formula:

\begin{greenformula}
	Let $\alpha \in \Omega^{p-1}(M)$ and $\beta \in \Omega^p(M)$.  Then
	\[
		\langle d\alpha , \beta \rangle_{L^2} - \langle \alpha , \delta \beta \rangle_{L^2} = \int_{\partial M} i^*\alpha \wedge i^*\!\star \beta.
	\]
\end{greenformula}

As in the closed case, the isomorphisms of $\hp_N(M)$ and $\hp_D(M)$ with the absolute and relative cohomology groups follow immediately from the decompositions \eqref{eqn:hmf1} and \eqref{eqn:hmf2} along with the de Rham and Duff theorems.  In particular, \eqref{eqn:hmf1} shows that $\hp_N(M)$ is the orthogonal complement of the exact $p$-forms inside the space of closed $p$-forms, so it is the concrete realization of $H^p(M; \mathbb{R})$.  Also, \eqref{eqn:hmf2} shows that $\hp_D(M)$ is the orthogonal complement of the space of relatively exact $p$-forms $\mathcal{E}^p_D(M)$ inside the space of closed relative $p$-forms (i.e. those satisfying the Dirichlet boundary condition), so it is the concrete realization of $H^p(M, \partial M; \mathbb{R})$.

A complete proof of the Hodge--Morrey--Friedrichs decomposition is given in Chapter 2 of Schwarz's book \cite{Schwarz}.


\subsubsection{Poincar\'e duality angles} 
\label{sub:poincar'e_duality_angles}
The Hodge--Morrey--Friedrichs Decomposition Theorem shows that there are concrete realizations $\hp_N(M)$ and $\hp_D(M)$ of the absolute and relative $p$th cohomology groups inside the space of $p$-forms.  In fact,
\[
	\hp_N(M) \cap \hp_D(M) = \{0\},
\]
so these concrete realizations meet only at the origin.  This follows from the strong unique continuation theorem of Aronszajn, Krzywicki and Szarski \cite{Aronszajn} (cf.\ \cite[Theorem 3.4.4]{Schwarz} for details).  But $\hp_N(M)$ and $\hp_D(M)$ are not orthogonal in general and so cannot both appear in the same orthogonal decomposition of $\Omega^p(M)$.  As DeTurck and Gluck note, the best that can be done is the following five-term decomposition, which is an immediate consequence of the Hodge--Morrey--Friedrichs decomposition:

\begin{theorem}\label{thm:dg3}
	Let $M$ be a compact, oriented, smooth Riemannian manifold with non-empty boundary.  Then the space $\Omega^p(M)$ of smooth $p$-forms on $M$ has the direct-sum decomposition
	\[
		\Omega^p(M) = c\mathcal{E}^p_N(M) \oplus \mathcal{E}c\mathcal{E}^p(M) \oplus \left(\hp_N(M) + \hp_D(M)\right) \oplus \mathcal{E}^p_D(M).
	\]
\end{theorem}
In the statement of \thmref{thm:dg3}, the symbol $\oplus$ indicates an orthogonal direct sum, whereas the symbol $+$ just indicates a direct sum.

\begin{proof}[Proof of \thmref{thm:dg3}]
	Since the spaces $\hp_N(M)$ and $\hp_D(M)$ meet only at the origin, the sum $\hp_N(M) + \hp_D(M)$ is direct.  Using the Hodge--Morrey--Friedrichs decomposition, the orthogonal complement of $\hp_N(M)$ inside $\hp(M)$ is $\mathcal{E}\hp(M)$, while the orthogonal complement of $\hp_D(M)$ inside $\hp(M)$ is $c\mathcal{E}\hp(M)$.  Hence, the orthogonal complement of $\hp_N(M) + \hp_D(M)$ inside $\hp(M)$ is
	\[
		\mathcal{E}\hp(M) \cap c\mathcal{E}\hp(M) = \mathcal{E}c\mathcal{E}^p(M).
	\]
	The rest of the proposed decomposition of $\Omega^p(M)$ is the same as in the Hodge--Morrey--Friedrichs Decomposition Theorem, so this completes the proof of \thmref{thm:dg3}.
\end{proof}

DeTurck and Gluck's key insight was that the non-orthogonality of $\hp_N(M)$ and $\hp_D(M)$ has to do with the fact that some of the cohomology of $M$ comes from the ``interior'' of $M$ and some comes from the boundary.  

In absolute cohomology, the interior subspace is very easy to identify.  Consider the map $i^*\!\!: H^p(M; \mathbb{R}) \to H^p(\partial M; \mathbb{R})$ induced by the inclusion $i: \partial M \to M$.  The kernel of $i^*$ certainly deserves to be called the interior portion of $H^p(M; \mathbb{R})$, but it is not clear what the boundary portion should be.  Since $H^p(M; \mathbb{R}) \cong \hp_N(M)$, the interior portion of the absolute cohomology is identifiable as the subspace of the harmonic Neumann fields which pull back to zero in the cohomology of the boundary; i.e.
\[
	\mathcal{E}_\partial \hp_N(M) := \{\omega \in \hp_N(M) : i^*\omega = d_\partial \varphi \text{ for some } \varphi \in \Omega^{p-1}(\partial M)\}.
\]

On the other hand, it is the boundary subspace which is easy to identify in relative cohomology.  Let $j: M = (M, \emptyset) \to (M, \partial M)$ be the inclusion and consider the long exact de Rham cohomology sequence of the pair $(M, \partial M)$
\begin{equation}\label{eqn:derhamexactsequence}
	\begin{diagram}
		\cdots  & \lTo & H^p_{\text{dR}}(\partial M) &\lTo^{\ \ i^*} & H^p_{\text{dR}}(M) &\lTo^{\ \ j^*} & H^p_{\text{dR}}(M, \partial M) &\lTo^{\ d} & H^{p-1}_{\text{dR}}(\partial M) &\lTo & \cdots.
	\end{diagram}
\end{equation}
The map $d: H^p_{\text{dR}}(\partial M) \to H^p_{\text{dR}}(M, \partial M)$ takes a closed $(p-1)$-form  $\varphi$ on $\partial M$, extends it arbitrarily to a $(p-1)$-form $\til{\varphi}$ on $M$, then defines $d[\varphi] := [d\til{\varphi}]$.  The form $d\til{\varphi}$ is certainly exact, but is not in general relatively exact.  Then the image of $H^{p-1}_{\text{dR}}(\partial M)$ inside $H^p_{\text{dR}}(M, \partial M)$ is the natural portion to interpret as coming from the boundary.

Translating \eqref{eqn:derhamexactsequence} into the notation of the Hodge--Morrey--Friedrichs decomposition gives the long exact sequence
\[
	\begin{diagram}
		\cdots  & \lTo & \hp(\partial M) &\lTo^{\ \ i^*} & \hp_N(M) &\lTo^{\ \ j^*} & \hp_D(M) &\lTo^{d} & \mathcal{H}^{p-1}(\partial M) &\lTo & \cdots
	\end{diagram}
\]
and the boundary subspace of $\hp_D(M)$ consists of those Dirichlet fields which are exact: the subspace $\mathcal{E}\hp_D(M)$.

The Hodge star takes $\mathcal{H}^{n-p}_D(M)$ to $\hp_N(M)$ and takes its boundary subspace $\mathcal{EH}^{n-p}_D(M)$ to $c\mathcal{E}\hp_N(M)$; regard this as the boundary subspace of $\hp_N(M)$.  Likewise, the interior subspace of $\hp_D(M)$ is identified as the image under the Hodge star of the interior subspace $\mathcal{E}_\partial \mathcal{H}^{n-p}_N(M)$  of $\mathcal{H}^{n-p}_N(M)$, namely the subspace
\[
	c\mathcal{E}_\partial\hp_D(M) := \{\omega \in \hp_D(M) : i^*\!\star \omega = d_\partial \psi \text{ for some } \psi \in \Omega^{n-p-1}(\partial M)\}.
\]

DeTurck and Gluck's refinement of the Hodge--Morrey--Friedrichs decomposition says that in both $\hp_N(M)$ and $\hp_D(M)$ the boundary subspace is the orthogonal complement of the interior subspace.

\begin{theorem}[DeTurck--Gluck]\label{thm:dg1}
	The spaces $\hp_N(M)$ and $\hp_D(M)$ admit the $L^2$-orthogonal decompositions into boundary and interior subspaces
	\begin{align}
		\label{eqn:dg1} \hp_N(M) & = c\mathcal{E}\hp_N(M) \oplus \mathcal{E}_\partial \hp_N(M) \\
		\label{eqn:dg2} \hp_D(M) & = \mathcal{E}\hp_D(M) \oplus c\mathcal{E}_\partial \hp_D(M).
	\end{align}
\end{theorem}
\begin{proof}
	It is sufficient to prove \eqref{eqn:dg1}, since substituting $n-p$ for $p$ in \eqref{eqn:dg1} and applying the Hodge star gives \eqref{eqn:dg2}.
	
	The fact that the two terms on the right hand side of \eqref{eqn:dg1} are orthogonal follows from Green's formula.  If $\delta \xi \in c\mathcal{E}\hp_N(M)$ and $\omega \in \mathcal{E}_\partial \hp_N(M)$, then $i^*\omega = d\varphi$ for some $\varphi \in \Omega^{p-1}(\partial M)$ and
	\[
		\langle \omega, \delta \xi \rangle_{L^2}  = \langle d\omega , \xi \rangle_{L^2} - \int_{\partial M} i^*\omega \wedge i^*\!\star \xi = - \int_{\partial M} d\varphi \wedge i^*\!\star \xi
	\]
	since $\omega$ is closed.  Extend $\varphi$ arbitrarily to some form $\til{\varphi} \in \Omega^{p-1}(M)$; then
	\[
		\langle \omega , \delta \xi \rangle_{L^2} = - \int_{\partial M} d\varphi \wedge i^*\!\star \xi = \langle dd\til{\varphi}, \xi \rangle_{L^2} - \int_{\partial M} d\varphi \wedge i^*\!\star \xi = \langle d\til{\varphi}, \delta \xi\rangle_{L^2}.
	\]
	Running Green's formula again, 
	\[
		\langle \omega , \delta \xi \rangle_{L^2} = \langle d\til{\varphi}, \delta \xi \rangle_{L^2} = \langle \til{\varphi}, \delta \delta \xi \rangle_{L^2} + \int_{\partial M} i^*\til{\varphi} \wedge i^*\!\star \delta \xi = 0
	\]
	since $\delta \xi$ is a Neumann form and hence $i^*\!\star \delta \xi = 0$.  Thus, the spaces $c\mathcal{E}\hp_N(M)$ and $\mathcal{E}_\partial \hp_N(M)$ are orthogonal.
	
	The proof that the sum of $c\mathcal{E}\hp_N(M)$ and $\mathcal{E}_\partial \hp_N(M)$ is equal to $\hp_N(M)$ proceeds in two steps.
	
	{\bf Step 1:} Let $c_1^p, \ldots , c_g^p$ be absolute $p$-cycles which form a homology basis for $i_* H_p(\partial M; \mathbb{R}) \subset H_p(M; \mathbb{R})$.  The goal is to show that there is a unique $p$-form in $c\mathcal{E}\hp_N(M)$ with preassigned periods $C_1, \ldots , C_g$ on these $p$-cycles.
	
	Extend the given homology basis for $i_*H_p(M; \mathbb{R})$ to a homology basis
	\[
		c_1^p, \ldots , c_g^p, c_{g+1}^p, \ldots , c_s^p
	\]
	for all of $H_p(M; \mathbb{R})$, where $s = b_p(M)$ is the $p$th Betti number of $M$.  Let
	\begin{equation}\label{eqn:homologybasis}
		c_1^{n-p}, \ldots , c_g^{n-p}, c_{g+1}^{n-p}, \ldots , c_s^{n-p}
	\end{equation}
	be a Poincar\'e dual basis for $H_{n-p}(M, \partial M; \mathbb{R})$, where $c_1^{n-p}, \ldots , c_g^{n-p}$ are relative $(n-p)$-cycles and $c_{g+1}^{n-p}, \ldots , c_s^{n-p}$ are absolute $(n-p)$-cycles.
	
	Since $\mathcal{H}^{n-p}_D(M) \cong H^{n-p}(M, \partial M; \mathbb{R})$, there is a unique form $\til{\eta} \in \mathcal{H}^{n-p}_D(M)$ with preassigned periods
	\[
		F_1, \ldots , F_g, F_{g+1}, \ldots , F_s
	\]
	on the basis \eqref{eqn:homologybasis}.  The form $\til{\eta}$ is exact (i.e. in the boundary subspace $\mathcal{EH}^{n-p}_D(M)$) precisely when
	\[
		F_{g+1} = \ldots = F_s = 0.
	\]
	Since the focus is on the boundary subspace, start with $\til{\eta} \in \mathcal{EH}^{n-p}_D(M)$ having periods
	\[
		F_1, \ldots , F_g, F_{g+1} =  \ldots  = F_s = 0.
	\]
	Let $\eta := \star\, \til{\eta} \in c\mathcal{E}\hp_N(M)$ and let $C_1, \ldots , C_g$ be the periods of $\omega$ on the $p$-cycles $c_1^p, \ldots , c_g^p$.
	
	To prove that there is an element of $c\mathcal{E}\hp_N(M)$ having arbitrary preassigned periods on $c_1^p, \ldots c_g^p$, it suffices to show that $(F_1, \ldots , F_g ) \mapsto (C_1, \ldots , C_g)$ is an isomorphism.
	
	Suppose some set of $F$-values gives all zero $C$-values, meaning that $i^*\eta$ is zero in the cohomology of $\partial M$.  In other words, the form $i^*\eta$ is exact, meaning that $\eta \in \mathcal{E}_\partial \hp_N(M)$, the interior subspace of $\hp_N(M)$.  Since $\mathcal{E}_\partial \hp_N(M)$ is orthogonal to $c\mathcal{E}\hp_N(M)$, this implies that $\eta = 0$, so $\til{\eta} = \pm \star \eta = 0$ and hence the periods $F_i$ of $\til{\eta}$ must have been zero.  
	
	Therefore, the map $(F_1, \ldots , F_g) \mapsto (C_1, \ldots , C_g)$ is an isomorphism, completing Step 1.
	
	{\bf Step 2:} Let $\omega \in \hp_N(M)$ and let $C_1, \ldots , C_g$ be the periods of $\omega$ on the above $p$-cycles $c_1^p, \ldots , c_g^p$.  Let $\alpha \in c\mathcal{E}\hp_N(M)$ be the unique form guaranteed by Step 1 having the same periods on this homology basis.
	
	Then $\beta = \omega - \alpha$ has zero periods on the $p$-cycles $c_1^p, \ldots , c_g^p$; since $\beta$ is a closed form on $M$, it certainly has zero period on each $p$-cycle of $\partial M$ which bounds in $M$.  Hence, $\beta$ has zero periods on all $p$-cycles of $\partial M$, meaning that $i^*\beta$ is exact, so $\beta \in \mathcal{E}_\partial \hp_N(M)$. 
	
	Therefore, $\omega = \alpha + \beta \in c\mathcal{E}\hp_N(M) + \mathcal{E}_\partial \hp_N(M)$, so $\hp_N(M)$ is indeed the sum of these two subspaces, as claimed in \eqref{eqn:dg1}.  This completes the proof of the theorem. 
\end{proof}

\thmref{thm:dg1} allows the details of Figure~\ref{fig:2planes1} to be filled in, as shown in Figure~\ref{fig:2planes2}.

\begin{figure}[htbp]
	\centering
		\includegraphics[scale=1]{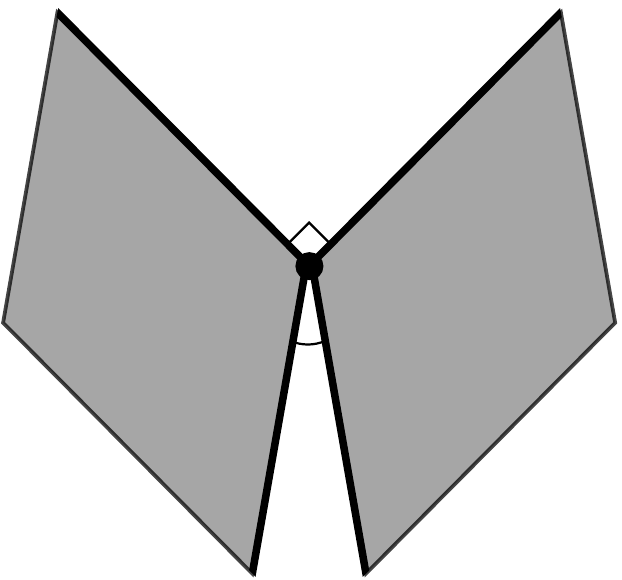}
		\put(-159,99){\rotatebox{-45}{\Large{$\hp_N(M)$}}}
		\put(-149,132){\rotatebox{-45}{{$c\mathcal{E}\hp_N(M)$}}}
		\put(-74,97){\rotatebox{45}{{$\mathcal{E}\hp_D(M)$}}}
		\put(-68,63){\rotatebox{45}{\Large{$\hp_D(M)$}}}
		\put(-117,25){\rotatebox{80}{{$\mathcal{E}_\partial \hp_N(M)$}}}
		\put(-85,75){\rotatebox{-79}{{$c\mathcal{E}_\partial \hp_D(M)$}}}
	\caption{$\hp_N(M)$ and $\hp_D(M)$}
	\label{fig:2planes2}
\end{figure}

With the interior and boundary subspaces given explicitly, DeTurck and Gluck's main theorem can now be stated:

\begin{theorem}[DeTurck--Gluck]\label{thm:dg2}
	Let $M^n$ be a compact, oriented, smooth Riemannian manifold with nonempty boundary $\partial M$.  Then within the space $\Omega^p(M)$ of $p$-forms on $M$,
	\begin{enumerate}
		\item \label{enum:dg1} The concrete realizations $\hp_N(M)$ and $\hp_D(M)$ of the absolute and relative cohomology groups $H^p(M; \mathbb{R})$ and $H^p(M, \partial M; \mathbb{R})$ meet only at the origin.
		\item \label{enum:dg2} The boundary subspace $c\mathcal{E}\hp_N(M)$ of $\hp_N(M)$ is orthogonal to all of $\hp_D(M)$ and the boundary subspace $\mathcal{E}\hp_D(M)$ of $\hp_D(M)$ is orthogonal to all of $\hp_N(M)$.
		\item \label{enum:dg3} No larger subspace of $\hp_N(M)$ is orthogonal to all of $\hp_D(M)$ and no larger subspace of $\hp_D(M)$ is orthogonal to all of $\hp_N(M)$.
		\item \label{enum:dg4} The principal angles between the interior subspaces $\mathcal{E}_\partial \hp_N(M)$ of $\hp_N(M)$ and $c\mathcal{E}_\partial \hp_D(M)$ of $\hp_D(M)$ are all acute.
	\end{enumerate}
\end{theorem}
The fact that $\mathcal{E}_\partial \hp_N(M)$ and $c\mathcal{E}_\partial \hp_D(M)$ have the same dimension is a straightforward consequence of Poincar\'e--Lefschetz duality.  The principal angles between these subspaces are invariants of the Riemannian structure on $M$ and are called the {\it Poincar\'e duality angles}.

\begin{proof} \makebox[1em]{}
	\begin{enumerate}
		\item As was already mentioned, the fact that  $\hp_N(M) \cap \hp_D(M) = \{0\}$ follows from the strong unique continuation theorem of Aronszajn, Krzywicki and Szarski.
		\item If $\delta \xi \in c\mathcal{E}\hp_N(M)$ and $\eta \in \hp_D(M)$, then, using Green's formula,
		\[
			\langle \eta, \delta \xi \rangle_{L^2} = \langle d\eta, \xi \rangle_{L^2} - \int_{\partial M} i^*\eta \wedge i^*\!\star \xi = 0
		\]
		since $\eta$ is closed and satisfies the Dirichlet boundary condition.  Therefore, the boundary subspace $c\mathcal{E}\hp_N(M)$ is orthogonal to all of $\hp_D(M)$.  
		
		Likewise, if $d\gamma \in \mathcal{E}\hp_D(M)$ and $\omega \in \hp_N(M)$, then
		\[
			\langle d\gamma, \omega \rangle_{L^2} = \langle \gamma , \delta \omega \rangle_{L^2} + \int_{\partial M} i^*\gamma \wedge i^*\!\star \omega = 0
		\]
		since $\omega$ is co-closed and satisfies the Neumann boundary condition.  Therefore, $\mathcal{E}\hp_D(M)$ is orthogonal to all of $\hp_N(M)$.
		\item This result follows from the Friedrichs decompositions \eqref{eqn:friedrichs} and \eqref{eqn:friedrichs2}, which said that
		\begin{align*}
			\hp(M) & = \hp_N(M) \oplus \mathcal{E}\hp(M)  \\
			& =  c\mathcal{E}\hp(M)  \oplus  \hp_D(M).
		\end{align*}
		If a form $\omega \in \hp_N(M)$ is orthogonal to all of $\hp_D(M)$, then it must be the case that $\omega \in c\mathcal{E}\hp(M)$ and, therefore, $\omega$ is in the boundary subspace $c\mathcal{E}\hp_N(M)$ of $\hp_N(M)$.
		
		Likewise, if $\eta \in \hp_D(M)$ is orthogonal to all of $\hp_N(M)$, then $\eta \in \mathcal{E}\hp(M)$ and therefore $\eta$ is in the boundary subspace $\mathcal{E}\hp_D(M)$.
		\item By (\ref{enum:dg1}) and (\ref{enum:dg3}), the principal angles between the interior subspaces can be neither 0 nor $\pi/2$, so they must all be acute.
	\end{enumerate}
\end{proof}

Suppose that $M$ is a Riemannian manifold with boundary and that $\theta_1^p, \ldots , \theta_k^p$ are the Poincar\'e duality angles in dimension $p$, i.e. $\theta_1^p, \ldots , \theta_k^p$ are the principal angles between the interior subspaces $\mathcal{E}_\partial \hp_N(M)$ and $c\mathcal{E}_\partial \hp_D(M)$.  If $\mathsf{proj}_D: \hp_N(M) \to \hp_D(M)$ is the orthogonal projection, then the images of the boundary and interior subspaces are $\mathsf{proj}_D c\mathcal{E}\hp_N(M) = 0$ and $\mathsf{proj}_D \mathcal{E}_\partial \hp_N(M) = c\mathcal{E}_\partial\hp_D(M)$.  Since the cosines of the principal angles between two $k$-planes are the singular values of the orthogonal projection from one to the other, the  $\cos \theta_i^p$ are the non-zero singular values of $\mathsf{proj}_D$.  Likewise, if $\mathsf{proj}_N: \hp_D(M) \to \hp_N(M)$ is the orthogonal projection, the $\cos \theta_i^p$ are also the non-zero singular values of $\mathsf{proj}_N$.

Thus, for $1\leq i \leq k$, the quantities $\cos^2 \theta_i^p$ are the non-zero eigenvalues of the compositions
\[
	\mathsf{proj}_N \circ \mathsf{proj}_D \quad \text{and} \quad \mathsf{proj}_D \circ \mathsf{proj}_N.
\]
It is this interpretation of the Poincar\'e duality angles which will yield the connection with the Dirichlet-to-Neumann map for differential forms given by \thmref{thm:pdangleseigenvalues}.



\clearpage 

\subsection{The Dirichlet-to-Neumann map} 
\label{sec:the_dirichlet_to_neumann_map}

\subsubsection{The classical Dirichlet-to-Neumann map and the problem of Electrical Impedance Tomography} 
\label{sub:classical_dn_map}

The Dirichlet-to-Neumann map for differential forms is a generalization of the classical Dirichlet-to-Neumann operator for functions.  The classical Dirichlet-to-Neumann operator arises in connection with the problem of Electrical Impedance Tomography (EIT), which was originally posed by Calder\'on \cite{Calderon} in the context of geoprospecting but which is also of interest in medical imaging (cf.\ \cite{Holder} for an overview of medical applications).

The problem of EIT is to determine the conductivity inside an open subset $\Omega \subset \mathbb{R}^3$ (or, more generally, $\Omega \subset \mathbb{R}^n$) by creating voltage potentials on the boundary and measuring the induced current flux through the boundary.  At low frequencies the electrical potential $u$ on $\Omega$ is governed by the Laplace equation
\begin{equation}\label{eqn:eit}
	\nabla \cdot \gamma \nabla u = 0,
\end{equation}
where $\gamma = (\gamma^{ij})$ is the positive-definite matrix giving the conductivity at points $x \in \Omega$ (cf.\ \cite[Appendix 1]{Cheney} for a derivation  of \eqref{eqn:eit} from Maxwell's equations).  The current flux through the boundary is the normal component of the current density at the boundary; denoting this by $j$, 
\[
	j = -\gamma \, \frac{\partial u}{\partial \nu},
\]
where $\nu$ is the unit inward-pointing normal vector.  If $f = u|_{\partial \Omega}$ is the electrical potential on the boundary, then the problem of EIT is to determine the conductivity $\gamma$ on all of $\Omega$ from the voltage-to-current map
\[
	f \mapsto -\gamma\, \frac{\partial u}{\partial \nu}.
\]

The classical Dirichlet-to-Neumann operator $\Lambda_{\text{cl}}: C^{\infty}(\partial \Omega) \to C^{\infty}(\partial \Omega)$ is defined by
\[
	f \mapsto \frac{\partial u}{\partial \nu},
\]
where $\Delta u = 0$ on $\Omega \subset \mathbb{R}^n$ and $u|_{\partial \Omega} = f$.  In dimension $\geq 3$, Lee and Uhlmann \cite{LeeUhlmann} showed that the problem of EIT is equivalent to determining an associated Riemannian metric $g$ from the Dirichlet-to-Neumann map $\Lambda_{\text{cl}}$, where
\[
	g_{ij} = \left(\det \gamma^{k\ell}\right)^{1/(n-2)}\left(\gamma^{ij}\right)^{-1}.
\]
This restatement of the problem of EIT in terms of the inverse problem of determining the Riemannian metric $g$ from the Dirichlet-to-Neumann map $\Lambda_{\text{cl}}$ makes sense on an arbitrary compact Riemannian manifold $M$ with non-empty boundary, regardless of whether or not $M$ is a region in Euclidean space, so this is the preferred mathematical formulation.

A prototypical theorem for this inverse problem is the following:

\begin{theorem}[Lee--Uhlmann]\label{thm:LeeUhlmann}
	If $M^n$ is a simply-connected, compact, real-analytic, geodesically convex Riemannian manifold with boundary and $n \geq 3$, then $(\partial M, \Lambda_{\mathrm{cl}})$ determines the Riemannian metric on $M$ up to isometry.  If ``real-analytic'' is replaced with ``smooth'', then $(\partial M, \Lambda_{\mathrm{cl}})$ determines the $C^{\infty}$-jet of the metric at the boundary of $M$.
\end{theorem}

Generalizations of the above theorem are given by Lassas and Uhlmann \cite{LassasUhlmann} and Lassas, Taylor, and Uhlmann \cite{LTU}.  

In dimension two, the problem of EIT and the problem of determining the Riemannian metric from $\Lambda_{\text{cl}}$ are distinct.  In this context the EIT problem for isotropic conductivities was solved by Nachman \cite{Nachman}; Sylvester \cite{Sylvester} showed that the anisotropic case reduces to the isotropic case.  Determining the metric from $\Lambda_{\text{cl}}$ is too much to ask for in dimension two, but Lassas and Uhlmann showed that the conformal class of a surface with boundary is determined by $\Lambda_{\text{cl}}$.


\subsubsection{The Dirichlet-to-Neumann map for differential forms} 
\label{sub:dn_map_forms}
The classical Dirichlet-to-Neumann map was generalized to differential forms independently by Joshi and Lionheart \cite{Joshi} and Belishev and Sharafutdinov \cite{Belishev}.  Their definitions are essentially equivalent, but  Belishev and Sharafutdinov's notation and definitions are used throughout this paper.

Let $M^n$ be a compact, oriented, smooth Riemannian manifold with non-empty boundary $\partial M$.  Define the Dirichlet-to-Neumann map for $p$-forms $\Lambda_p: \Omega^{p}(\partial M) \to \Omega^{n-p-1}(\partial M)$ for any $0 \leq p \leq n-1$ as follows.  

If $\varphi \in \Omega^{p}(\partial M)$ is a smooth $p$-form on the boundary, then the boundary value problem 
\begin{equation}\label{eqn:bvp}
	\Delta \omega = 0, \quad i^*\omega = \varphi \quad \text{and} \quad i^*\delta \omega = 0
\end{equation}
can be solved (cf.\ \cite[Lemma 3.4.7]{Schwarz}).  The solution $\omega \in \Omega^p(M)$ is unique up to the addition of an arbitrary harmonic Dirichlet field $\lambda \in \hp_D(M)$.  Define
\[
	\Lambda_p \varphi := i^*\!\star d\omega.
\]
Then $\Lambda_p \varphi$ is independent of the choice of $\omega$ since taking $d\omega$ eliminates the ambiguity in the choice of $\omega$.  Define 
\[
	\Lambda := \bigoplus_{i=0}^{n-1} \Lambda_i.
\]

When $\varphi$ is a function (i.e. $\varphi \in \Omega^0(\partial M)$), suppose $u \in \Omega^0(M)$ is a harmonic function which restricts to $\varphi$ on the boundary.  Since $\delta u = 0$, $u$ solves the boundary value problem \eqref{eqn:bvp}.  Hence,
\[
	\Lambda_0 \varphi = i^*\!\star u = \frac{\partial u}{\partial \nu} \, \mathrm{dvol}_{\partial M} = \left(\Lambda_{\text{cl}}\varphi \right) \text{dvol}_{\partial M},
\]
so $\Lambda$ is indeed a generalization of the classical Dirichlet-to-Neumann map.

In the spirit of \thmref{thm:LeeUhlmann}, Joshi and Lionheart showed that the Dirichlet-to-Neumann map for differential forms recovers information about the metric on $M$:

\begin{theorem}[Joshi--Lionheart]\label{thm:JoshiLionheart}
	For any $p$ such that $0 \leq p \leq n-1$, the data $(\partial M, \Lambda_p)$ determines the $C^\infty$-jet of the Riemannian metric at the boundary of $M$.
\end{theorem}

Belishev and Sharafutdinov take a more topological approach, looking to determine the cohomology of $M$ from the Dirichlet-to-Neumann map.  Two key lemmas, both in their argument and for Section~\ref{chap:dnmap}, are the following:

\begin{lemma}[Belishev--Sharafutdinov]\label{lem:bvpharmonic}
	If $\varphi \in \Omega^p(\partial M)$ and $\omega \in \Omega^p(M)$ solves the boundary value problem \eqref{eqn:bvp}, then $d\omega \in \mathcal{H}^{p+1}(M)$ and $\delta \omega = 0$.  Hence, \eqref{eqn:bvp} is equivalent to the boundary value problem
	\begin{equation}\label{eqn:bvp1}
		\Delta \omega = 0, \quad i^*\omega = \varphi \quad \text{and} \quad \delta \omega = 0.
	\end{equation}
\end{lemma}

\begin{lemma}[Belishev--Sharafutdinov]\label{lem:harmonictraces}
	For any $0 \leq p \leq n-1$, the kernel of $\Lambda_p$ coincides with the image of $\Lambda_{n-p-1}$.  Moreover, a form $\varphi \in \Omega^{p}$ belongs to $\ker \Lambda_p = \mathrm{im }\, \Lambda_{n-p-1}$ if and only if $\varphi = i^* \omega$ for some harmonic field $\omega \in \hp(M)$.  In other words,
	\[
		i^*\hp(M) = \ker \Lambda_p = \mathrm{im }\, \Lambda_{n-p-1}.
	\]
\end{lemma}

Knowledge of the kernel of the Dirichlet-to-Neumann map yields lower bounds on the Betti numbers $b_p(M)$ of $M$ and $b_p(\partial M)$ of $\partial M$:

\begin{theorem}[Belishev--Sharafutdinov]\label{thm:bs1}
	The kernel $\ker \Lambda_p$ of the Dirichlet-to-Neumann map $\Lambda_p$ contains the space $\mathcal{E}^p(\partial M)$ of exact $p$-forms on $\partial M$ and
	\[
		\dim \left[ \ker \Lambda_p / \mathcal{E}^p(\partial M) \right] \leq \min \left \{ b_p( M), b_p(\partial M) \right \}.
	\]
\end{theorem}

\thmref{thm:kernel} provides a refinement of this theorem.  

Two other operators come to attention in Belishev and Sharafutdinov's story.  The first is the {\it Hilbert transform} $T$, defined as $T := d_\partial\Lambda^{-1}$.  The Hilbert transform is obviously not well-defined on all forms on $\partial M$, but is well-defined on $i^*\hp(M) = \text{im } \Lambda_{n-p-1}$ for any $0 \leq p \leq n-1$.  The analogy between the map $T$ and the classical Hilbert transform from complex analysis is explained in Belishev and Sharafutdinov's Section 5.  

The other interesting operator $G_p: \Omega^p(\partial M) \to \Omega^{n-p-1}(\partial M)$ is defined as
\[
	G_p := \Lambda_p + (-1)^{pn+p+n} d_\partial\Lambda_{n-p-2}^{-1} d_\partial.
\]
Letting $G = \bigoplus_{i=0}^{n-1} G_i$, note that $G = \Lambda \pm T d_\partial$.

Belishev and Sharafutdinov's main theorem shows that knowledge of $\Lambda$ (and thus of $G$) yields knowledge of the cohomology of $M$:

\pagebreak

\begin{theorem}[Belishev--Sharafutdinov]\label{thm:bs2}
	For any $0 \leq p \leq n-1$, 
	\[
		\mathrm{im }\, G_{n-p-1} = i^*\hp_N(M).
	\]
	Since harmonic Neumann fields are uniquely determined by their pullbacks to the boundary, this means that $\mathrm{im }\, G_{n-p-1} \cong \hp_N(M)  \cong H^p(M; \mathbb{R})$.  In other words, the boundary data $(\partial M, \Lambda)$ completely determines the absolute cohomology groups of $M$.
\end{theorem}

By Poincar\'e--Lefschetz duality, $H^p(M; \mathbb{R}) \cong H^{n-p}(M, \partial M; \mathbb{R})$, so the above theorem immediately implies that $(\partial M, \Lambda)$ also determines the relative cohomology groups of $M$.

A key feature of \thmref{thm:bs2} is that the cohomology groups $H^p(M; \mathbb{R})$ and $H^p(M, \partial M; \mathbb{R})$ are not just determined abstractly by $(\partial M, \Lambda)$, but can be realized as particular subspaces of differential forms on $\partial M$.  The content of \thmref{thm:pdangleseigenvalues} is that these shadows of $H^p(M; \mathbb{R})$ and $H^p(M, \partial M; \mathbb{R})$ on the boundary detect the relative positions of the spaces $\hp_N(M)$ and $\hp_D(M)$.



\section{Poincar\'e duality angles on complex projective space}
\label{chap:pdangles_cpn}

The simplest example of a compact manifold with boundary which is intuitively close to being closed is one which is obtained from a closed manifold by removing a small ball.  Therefore, determining the Poincar\'e duality angles for an interesting class of such manifolds should provide insight into whether the Poincar\'e duality angles measure how close a manifold is to being closed.

To that end, for $n \geq 2$ consider the complex projective space $\CP^n$ with its usual Fubini-Study metric and let $x \in \CP^n$.  For $0 < r < \pi/2$, define the one-parameter family of compact Riemannian manifolds with boundary
\[
	M_r := \CP^n - B_r(x),
\]
where $B_r(x)$ is the open ball of radius $r$ centered at $x$.  The cohomology groups of $M_r$ are
\begin{equation}\label{eqn:cpncohomology}
	H^i(M_r; \mathbb{R}) \cong H^{2n-i}(M_r, \partial M_r; \mathbb{R})  = \begin{cases} \mathbb{R} & \text{for } i = 2k, \ 0 \leq k \leq n-1 \\ 0 & \text{otherwise.}\end{cases}
\end{equation}
The goal of this section is to prove \thmref{thm:cpn}:

\begin{theoremcpn}
	For $1 \leq k \leq n-1$ there is a non-trivial Poincar\'e duality angle $\theta_r^{2k}$ between the concrete realizations of $H^{2k}(M_r; \mathbb{R})$ and $H^{2k}(M_r, \partial M_r ; \mathbb{R})$ which is given by
	\begin{equation}\label{eqn:cpn}
		\cos \theta_r^{2k} = \frac{1-\sin^{2n}r}{\sqrt{(1+\sin^{2n}r)^2 + \frac{(n-2k)^2}{k(n-k)}\sin^{2n}r}}.
	\end{equation}
\end{theoremcpn}

For small $r$, $\cos \theta_r^{2k} = 1 + O(r^{2n})$ and $2n = \dim \CP^n = \text{codim } \{x\}$, so the Poincar\'e duality angles not only go to zero as $r \to 0$, but seem to detect the codimension of the point removed.  Also, for $r$ near its maximum value of $\pi/2$, $\cos \theta_r^{2k} = O(r^2)$ and $2 = \text{codim } \CP^{n-1}$, which is the manifold onto which $M_r$ collapses as $r \to \pi/2$.

In order to compute these angles it is necessary to get a detailed geometric picture of $M_r$.

\subsection{The geometric situation} 
\label{sec:cpngeometry}
The cut locus of the point $x \in \CP^n$ is a copy of $\CP^{n-1}$ sitting at constant distance $\pi/2$ from $x$.  Thus, thinking of $x$ and this copy of $\CP^{n-1}$ as sitting at opposite ``ends'' of $\CP^n$, the space in between the two is foliated by copies of $S^{2n-1}$ given by exponentiating the concentric spheres around $x$ in $T_x \CP^n$.

This follows simply from the definition of the cut locus, but can also be seen as follows. There is a Hopf fibration $H: S^{2n+1} \to \CP^{n}$ given by identifying points on the same complex line.  The preimage $H^{-1}(x)$ is a copy of $S^1$, the preimage $H^{-1}(\CP^{n-1})$ is a copy of $S^{2n-1}$ and $S^{2n+1}$ can be viewed as the join of these copies of $S^1$ and $S^{2n-1}$ in the usual way.  Hence, the region between these spheres is foliated by copies of $S^1 \times S^{2n-1}$; metrically, the leaves are given by
\[
	S^{2n-1}(\cos t) \times S^1(\sin t),
\] 
where $t \in (0, \pi/2)$ measures the distance from $S^{2n-1}$.  Then the hypersurface $H\left(S^{2n-1}(\cos t) \times S^1(\sin t)\right)$ sits at constant distance $t$ from $\CP^{n-1}$ (and thus distance $\pi/2 - t$ from $x$).  Since the horizontal spheres $S^{2n-1}(\cos t) \times \{\text{pt}\}$ meet each Hopf fiber exactly once, this hypersurface is topologically $S^{2n-1}$.  

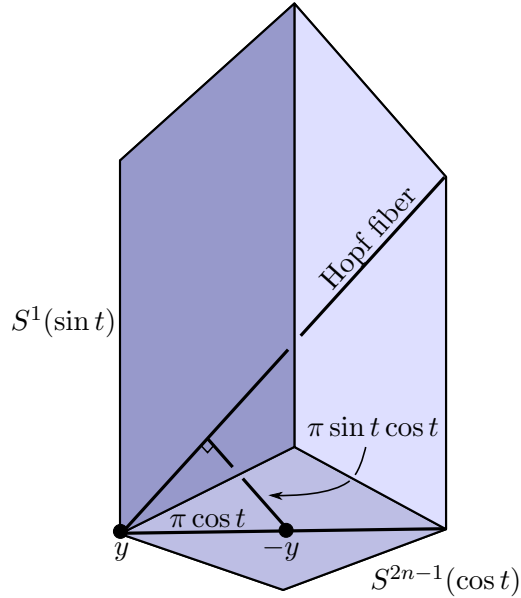
\begin{figure}[htbp]
	\centering
		\includegraphics[scale=1]{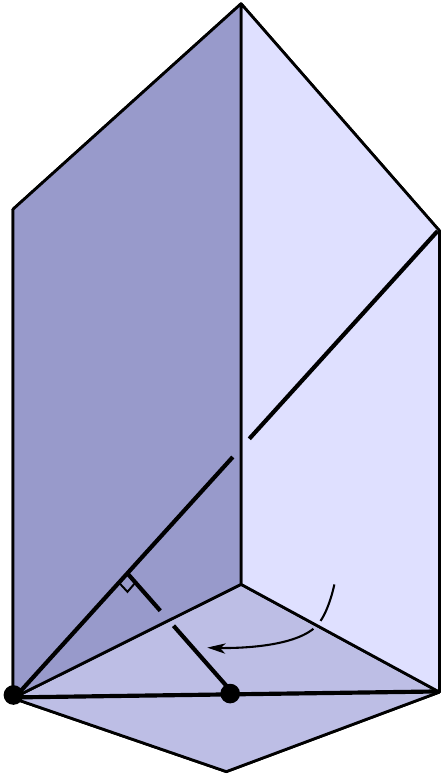}
	\put(-106,25.5){$\pi \cos t$}
	\put(-166,100){$S^1(\sin t)$}
	\put(-30, 2){$S^{2n-1}(\cos t)$}
	\put(-54, 60){$\pi \sin t \cos t$}
	\put(-50, 115){\rotatebox{48}{Hopf fiber}}
	\put(-127,15){$y$}
	\put(-71,15){$-y$}
	\caption{$S^{2n-1}(\cos t) \times S^1(\sin t)$ sitting over the hypersurface at distance $t$ from $\CP^{n-1}$}
	\label{fig:hopffibration}
\end{figure}

Metrically, $H\left(S^{2n-1}(\cos t) \times S^1(\sin t)\right)$ is a round sphere which has been scaled by a factor of $\cos t$ and whose Hopf fibers have been scaled by an additional factor of $\sin t$.  This is illustrated in Figure~\ref{fig:hopffibration}, where the ``big diagonal'' is a Hopf fiber in $S^{2n+1}$ and the ``little diagonal'' is a Hopf circle in the horizontal $S^{2n-1}(\cos t)$.  The distance between antipodal points $y$ and $-y$ on the horizontal Hopf circle is $\pi \cos t$.  However, the minimum distance from  $-y$ to the Hopf fiber in $S^{2n-1}(\cos t) \times S^1(\sin t)$ containing $y$ is $\pi \sin t \cos t$, so in the quotient the Hopf direction is scaled by an additional factor of $\sin t$.

Therefore, the hypersurface at distance $t$ from $\CP^{n-1}$ is the Berger sphere
\[
	S^{2n-1}(\cos t)_{\sin t}
\]
which is obtained from the unit sphere by scaling the entire sphere by $\cos t$ and the Hopf fibers by an additional factor of $\sin t$.

Since the manifold $M_r$ is just $\CP^n - B_r(x)$, it consists of all the spheres $S^{2n-1}(\cos t)_{\sin t}$ for $0 < t \leq \pi/2-r$ along with $\CP^{n-1}$ with its usual metric (which, with an abuse of notation, could be identified with $S^{2n-1}(\cos 0)_{\sin 0}$).  Hence, $M_r$ is a 2-disk bundle over $\CP^{n-1}$ and so has the same absolute cohomology as $\CP^{n-1}$ since the fibers are contractible.  This justifies the statement of the cohomology groups of $M_r$ given in \eqref{eqn:cpncohomology}.

Moreover, since $\partial M_r$ is homeomorphic to $S^{2n-1}$, none of the cohomology in dimensions strictly between 0 and $2n$ comes from the boundary, meaning that there are non-trivial Poincar\'e duality angles between the concrete realizations of $H^{2k}(M; \mathbb{R})$ and $H^{2k}(M, \partial M; \mathbb{R})$ for all $1 \leq k \leq n-1$.


\subsection{Finding harmonic fields} 
\label{sec:finding_harmonic_forms}
The goal of this section is to find harmonic $2k$-fields satisfying Neumann and Dirichlet boundary conditions for each $1 \leq k \leq n-1$.  Finding such harmonic fields and measuring the angles between them is sufficient to determine the Poincar\'e duality angles since $H^{2k}(M; \mathbb{R})$ and $H^{2k}(M, \partial M; \mathbb{R})$ are both 1-dimensional with no boundary subspace for each such $k$.

These harmonic fields must be isometry-invariant\,---\,otherwise they could be averaged over the action of the isometry group to get isometry-invariant forms representing the same cohomology class.  But this averaging does not affect whether the form is closed or co-closed, so the harmonic fields must have been isometry-invariant to start with.  

Any isometry of $M_r$ must map the hypersurface $S^{2n-1}(\cos t)_{\sin t}$ at constant distance $\pi/2 - r-t$ from the boundary to itself and must preserve the Hopf fibers on this hypersurface since they are scaled differently from the other directions.  Equivalently, any isometry of $M_r$ extends to an isometry of $\CP^n$ which fixes the point $x$, so the isometry group of $M_r$ is just the isotropy subgroup of $\CP^n$ (the identity component of which is $SU(n)$).  In either formulation, the Hopf direction and the $t$ direction are invariant directions and the $(2n-2)$-plane distribution orthogonal to both is also invariant.

With this in mind, let  $h_t:  S^{2n-1} \to S^{2n-1}(\cos t)_{\sin t}$ be the obvious diffeomorphism of the unit sphere with the hypersurface at constant distance $t$ from $\CP^{n-1}$ and let $H: S^{2n-1} \to \CP^{n-1}$ be the Hopf fibration.  Let $A$ be the vector field on $\CP^n$ which restricts on each hypersurface to the push-forward by $h_t$ of the unit vector field in the Hopf direction on $S^{2n-1}$.  Define $\alpha$ to be the 1-form dual to $A$ and let $\tau = dt$ be the 1-form dual to the $t$ direction. Define $\eta$ to be the 2-form on $\CP^n$ which restricts on each $S^{2n-1}(\cos t)_{\sin t}$ to
\[
	(H \circ h_t^{-1})^* \eta_{\CP^{n-1}},
\]
where $\eta_{\CP^{n-1}}$ denotes the standard symplectic form on $\CP^{n-1}$.  Then $\eta$ is a symplectic form on the $(2n-2)$-plane distribution orthogonal to both $A$ and $\frac{\partial}{\partial t}$.

Away from $\CP^{n-1}$ the manifold $M_r$ is topologically a product $S^{2n-1} \times I$, so exterior derivatives can be computed as in $S^{2n-1}$.  Thus
\[
	d\alpha = -2\eta
\]
and $\tau$ and $\eta$ are closed.  

The goal is to use the above information to construct closed and co-closed $2k$-forms $\omega_N$ and $\omega_D$ satisfying Neumann and Dirichlet conditions, respectively.  Since such harmonic fields must be isometry-invariant and thus map $S^{2n-1}(\cos t)_{\sin t}$ to itself and preserve the $A$ and $\frac{\partial}{\partial t}$ directions, they  must be of the form
\begin{align}
\label{eqn:cpnomegan}	\omega_N & := f_N(t)\, \eta^k + g_N(t)\, \alpha \wedge \eta^{k-1} \wedge \tau \\
\nonumber	\omega_D & := f_D(t)\, \eta^k + g_D(t)\, \alpha \wedge \eta^{k-1} \wedge \tau.
\end{align}

The requirement that $\omega_N$ be closed means that
\begin{align*}
	0  = d\omega_N & = d(f_N(t)\, \eta^k) + d(g_N(t)\, \alpha \wedge \eta^{k-1} \wedge \tau) \\
	& = f_N'(t)\, \eta^k \wedge \tau - 2g_N(t)\, \eta^k \wedge \tau,
\end{align*}
meaning that $0 = f_N'(t) - 2g_N(t)$.  Hence, $\omega_N$ is closed if and only if
\begin{equation}\label{eqn:cpnclosed}
	g_N(t) = \frac{1}{2}f_N'(t).
\end{equation}
Likewise, $\omega_D$ is closed if and only if $g_D(t) = \frac{1}{2}f_D'(t)$.

Since the exterior derivative is topological none of the above depended on the metric, but the Hodge star depends fundamentally on the metric.  The vector field $A$ is not a unit vector field; it is dual to the Hopf direction on $S^{2n-1}(\cos t)_{\sin t}$, which is scaled by a factor of $\sin t \cos t$ from the Hopf direction on the unit sphere.  Therefore, the unit vector field in the direction of $A$ is given by $\frac{1}{\sin t \cos t}A$ and the dual $1$-form is
\[
	\sin t \cos t\, \alpha.
\]
The $2$-form $\eta$ is dual to a $2$-plane distribution tangent to $S^{2n-1}(\cos t)_{\sin t}$ but orthogonal to $A$.  Such directions are scaled by a factor of $\cos t$.  If those directions are normalized, the dual 2-form must be $\cos^2 t \, \eta$.  Finally, $\frac{\partial}{\partial t}$ \textit{is} a unit vector field, so $\tau$ is already normalized.

The volume form on $M_r$ is, therefore, 
\[
	\frac{1}{(n-1)!}\,\sin t \cos^{2n-1} t\ \alpha \wedge \eta^{n-1} \wedge \tau,
\]
so the relevant computations of the Hodge star are
\[
	\star \, \left( \frac{1}{(k-1)!}\sin t \cos^{2k-1}t\ \alpha \wedge \eta^{k-1} \wedge \tau\right)  = \frac{1}{(n-k)!}\cos^{2n-2k}t\ \eta^{n-k}
\]
and
\[
	\star \, \left( \frac{1}{k!} \cos^{2k} t\ \eta^k\right) = \frac{1}{(n-k-1)!} \sin t \cos^{2n-2k-1}t\ \alpha \wedge \eta^{n-k-1} \wedge \tau.
\]
Combining this with \eqref{eqn:cpnomegan} yields 
\begin{align}
\nonumber	\star\, \omega_N & = \frac{(k-1)!}{(n-k-1)!}\left[\vphantom{\frac{\cos^{2n-4k+1}t}{\sin t}g_N(t)}k \sin t \cos^{2n-4k-1}t\, f_N(t)\, \alpha \wedge \eta^{n-k-1} \wedge \tau \right. \\
& \qquad \qquad \qquad \qquad \left. + \frac{1}{n-k} \frac{\cos^{2n-4k+1}t}{\sin t}g_N(t)\, \eta^{n-k}\right],
\label{eqn:cpnstaromegan}
\end{align}
so
\begin{align*}
	d\star \omega_N & = \frac{(k-1)!}{(n-k-1)!}\left[-2k \sin t \cos^{2n-4k-1}t\, f_N(t) \vphantom{\frac{\cos^n}{\sin}}\right. \\ & \qquad \qquad \qquad \qquad  - \frac{1}{n-k} \left(\left((2n-4k+1)\cos^{2n-4k}t + \frac{\cos^{2n-4k+2}t}{\sin^2 t}\right)g_N(t)\right. \\
	& \qquad \qquad \qquad \qquad \left.\left. - \frac{\cos^{2n-4k+1}t}{\sin t} g_N'(t)\right)\right] \eta^{n-k}\wedge \tau.
\end{align*}
Using \eqref{eqn:cpnclosed} and simplifying, this implies that $\omega_N$ being co-closed is equivalent to $f_N(t)$ satisfying the ODE
\begin{equation}\label{eqn:cpnode}
	0 = f_N''(t) - \left((2n-4k+1) \tan t + \cot t\right)f_N'(t) - 4k(n-k)\tan^2 t\, f_N(t).
\end{equation}

Solutions of this ODE take the form
\[
	f_N(t) = C_1 \cos^{2k} t + C_2 \frac{1}{\cos^{2n-2k}t},
\]
yielding
\[
	g_N(t) = \frac{1}{2}f_N'(t) = -kC_1 \sin t \cos^{2k-1}t + (n-k)C_2 \frac{\sin t}{\cos^{2n-2k+1}t}.
\]

The boundary of $M_r$ occurs at $t = \pi/2 - r$, so, using \eqref{eqn:cpnstaromegan}, 
\[
	i^*\!\star \omega_N = \frac{(k-1)!}{(n-k)!}\frac{\cos^{2n-4k+1}(\pi/2-r)}{\sin (\pi/2-r)}g_N(\pi/2-r)\, \eta^{n-k}.
\]
Therefore, $\omega_N$ satisfies the Neumann boundary condition $i^*\!\star \omega_N = 0$ if and only if
\[
	0 = g_N(\pi/2-r) = -kC_1 \cos r \sin^{2k-1}r + (n-k)C_2 \frac{\cos r}{\sin^{2n-2k+1}r},
\]
meaning that
\[
	C_2 = \frac{k}{n-k}C_1 \sin^{2n}r.
\]
Renaming $C_1$ as $C_N$, this implies that
\begin{align}
\label{eqn:cpnfngn}	f_N(t) & = C_N\left[ \cos^{2k} t + \frac{k}{n-k} \sin^{2n}r\, \frac{1}{\cos^{2n-2k}t}\right] \\
\nonumber	g_N(t) & = C_N\left[ -k \sin t \cos^{2k-1}t + k \sin^{2n}r\, \frac{\sin t}{\cos^{2n-2k+1}t}\right].
\end{align}

Since $\omega_D$ must be closed and co-closed, $f_D$ and $g_D$ also satisfy the equations \eqref{eqn:cpnclosed} and \eqref{eqn:cpnode} and so take the same basic form as $f_N$ and $g_N$.  The Dirichlet boundary condition $i^*\omega_D = 0$ implies that $f_D(\pi/2-r) = 0$, so
\begin{align}
	\label{eqn:cpnfdgd} f_D(t) & = C_D\left[\cos^{2k}t - \sin^{2n}r\, \frac{1}{\cos^{2n-2k}t}\right] \\
	\nonumber g_D(t) & = C_D \left[ -k\sin t \cos^{2k-1}t - (n-k) \sin^{2n}r\, \frac{\sin t}{\cos^{2n-2k+1}t} \right]
\end{align}
for some constant $C_D$.


\subsection{Normalizing the forms} 
\label{sec:normalizing_the_forms}
The angle $\theta_r^{2k}$ between $\omega_N$ and $\omega_D$ is given by
\[
	\langle \omega_N, \omega_D \rangle_{L^2} = \|\omega_N\|_{L^2}\, \|\omega_D\|_{L^2} \cos \theta_r^{2k}.
\]
Therefore, the constants $C_N$ and $C_D$ should be chosen such that
\[
	\|\omega_N\|_{L^2}  = 1 = \|\omega_D\|_{L^2}.
\]

Using \eqref{eqn:cpnomegan} and \eqref{eqn:cpnstaromegan}, 
\begin{align*}
	\omega_N \wedge \star\, \omega_N & = \frac{(k-1)!}{(n-k-1)!}\left[\vphantom{\frac{\cos^{2n-4k+1}t}{\sin t}g_N(t)^2}k \sin t \cos^{2n-4k+1}t\, f_N(t)^2  \right.\\
	& \qquad \qquad \qquad \quad \left. + \frac{1}{n-k} \frac{\cos^{2n-4k+1}t}{\sin t}g_N(t)^2 \right]\alpha \wedge \eta^{n-1}\wedge \tau.
\end{align*}

Thus,
\begin{align*}
	\langle \omega_N, \omega_N\rangle_{L^2} & = \int_{M_r} \omega_N \wedge \star\, \omega_N \\
	& = \int_{M_r} \frac{(k-1)!}{(n-k-1)!}\left[\vphantom{\frac{\cos^{2n-4k+1}t}{\sin t}} k \sin t \cos^{2n-4k+1}t\, f_N(t)^2 \right. \\
	& \qquad \qquad \qquad \qquad  \quad \left. + \frac{1}{n-k} \frac{\cos^{2n-4k+1}t}{\sin t}g_N(t)^2 \right]\alpha \wedge \eta^{n-1}\wedge \tau.
\end{align*}
$M_r$ is a product away from $\CP^{n-1}$ (which has measure zero), so Fubini's Theorem implies that $\langle \omega_N, \omega_N\rangle_{L^2}$ is given by
\begin{multline}\label{eqn:cpnfubini}
	\int_{0}^{\pi/2-r} \left[ \int_{S^{2n-1}(\cos t)_{\sin t}} \alpha \wedge \eta^{n-1}\right] \frac{(k-1)!}{(n-k-1)!}\left(k \sin t \cos^{2n-4k+1}t f_N(t)^2 \vphantom{\frac{\cos^n}{\sin}}\right. \\
	\qquad\qquad\qquad\qquad\qquad\qquad\qquad\qquad\qquad\qquad\quad\left. + \frac{1}{n-k} \frac{\cos^{2n-4k+1}t}{\sin t}g_N(t)^2 \right)dt.
\end{multline}
Since $\frac{1}{(n-1)!} \sin t \cos^{2n-1}t\, \alpha \wedge \eta^{n-1}$ is the volume form on $S^{2n-1}(\cos t)_{\sin t}$ and since
\[
	\text{vol}\,S^{2n-1}(\cos t)_{\sin t} = \sin t \cos^{2n-1}t\ \text{vol}\, S^{2n-1}(1),
\]
the expression \eqref{eqn:cpnfubini} reduces to
\[
	\langle \omega_N, \omega_N \rangle_{L^2} = Q \int_0^{\pi/2-r}\left[k \sin t \cos^{2n-4k+1}t\, f_N(t)^2  + \frac{1}{n-k} \frac{\cos^{2n-4k+1}t}{\sin t}g_N(t)^2 \right]dt,
\]
where 
\begin{equation}\label{eqn:cpnq}
	Q := \frac{(n-1)!(k-1)!}{(n-k-1)!}\,\text{vol}\,S^{2n-1} (1)
\end{equation}

Using the expressions for $f_N(t)$ and $g_N(t)$ given in \eqref{eqn:cpnfngn} a straightforward computation yields
\begin{align*}
	\langle \omega_N, \omega_N \rangle_{L^2} & = C_N^2 Q \int_0^{\pi/2-r} \left[\frac{kn}{n-k}\sin t \cos^{2n-1}t + \frac{k^2n}{(n-k)^2}\sin^{4n}r\, \frac{\sin t}{\cos^{2n+1}t}\right]dt \\
	& = C_N^2 Q\frac{k}{2(n-k)}\left(1+\frac{2k-n}{n-k}\sin^{2n}r - \frac{k}{n-k}\sin^{4n}r\right) \\
	& = C_N^2 Q \frac{k}{2(n-k)}\left(1+\frac{k}{n-k}\sin^{2n}r\right)\left(1-\sin^{2n}r\right).
\end{align*}
Therefore, the condition $\|\omega_N\|_{L^2} = 1$ means that $C_N$ should be chosen as
\begin{equation}\label{eqn:cpncn}
	C_N = \sqrt{\frac{2(n-k)!}{\text{vol}\,S^{2n-1}(1)\, (n-1)!\, k!\left(1+\frac{k}{n-k}\sin^{2n}r\right)\left(1-\sin^{2n}r\right)}}.
\end{equation}

An entirely analogous calculation shows that $\|\omega_D\|_{L^2} = 1$ when
\begin{equation}\label{eqn:cpncd}
	C_D = \sqrt{\frac{2(n-k)!}{\text{vol}\,S^{2n-1}(1)\, (n-1)!\,(k-1)! \left(\frac{k}{n-k} + \sin^{2n}r\right)\left(1-\sin^{2n}r\right)}}.
\end{equation}


\subsection{Computing the Poincar\'e duality angle} 
\label{sec:cpn_computing_the_poincar'e_duality_angle}
Now that the forms $\omega_N$ and $\omega_D$ are completely pinned down, computing the angle between them is straightforward:
\[
	\cos \theta_r^{2k} = \langle \omega_N, \omega_D \rangle_{L^2}
\]
since $\omega_N$ and $\omega_D$ have unit norm.

Using \eqref{eqn:cpnomegan} and the modification of \eqref{eqn:cpnstaromegan} appropriate for $\star\, \omega_D$, 
\begin{align*}
	\omega_N \wedge \star\, \omega_D & = \frac{(k-1)!}{(n-k-1)!}\left[ k \sin t \cos^{2n-4k-1}t\, f_N(t)f_D(t) \vphantom{\frac{\cos^n}{\sin}} \right. \\
	& \qquad \qquad \qquad \quad \ \left. + \frac{1}{n-k} \frac{\cos^{2n-4k+1}t}{\sin t} g_N(t) g_D(t)\right] \alpha \wedge \eta^{n-1} \wedge \tau.
\end{align*}
Therefore, using the same trick to integrate out $S^{2n-1}(\cos t)_{\sin t}$ as in \eqref{eqn:cpnfubini} and the ensuing few lines,
\begin{align*}
	\cos \theta_r^{2k} & = Q \int_{0}^{\pi/2-r} \left[\vphantom{\frac{\cos^{2n-4k+1}t}{\sin t}} k \sin t \cos^{2n-4k-1}t\, f_N(t)f_D(t) \right.\\
	& \qquad \qquad \qquad \left. + \frac{1}{n-k} \frac{\cos^{2n-4k+1}t}{\sin t} g_N(t) g_D(t)\right] dt
\end{align*}

The definitions of $f_N(t), f_D(t), g_N(t)$ and $g_D(t)$ then yield
\begin{align}
\nonumber	\cos \theta_r^{2k} & = C_NC_DQ \int_0^{\pi/2-r} \left[\frac{kn}{n-k}\sin t \cos^{2n-1}t - \frac{kn}{n-k}\sin^{4n}r\, \frac{\sin t}{\cos^{2n+1}t}\right] dt \\
\label{eqn:cpntheta1}	& = C_NC_DQ\frac{k}{2(n-k)}\left(1-\sin^{2n}r\right)^2.
\end{align}

The expressions for $Q$, $C_N$ and $C_D$ given in \eqref{eqn:cpnq}, \eqref{eqn:cpncn} and \eqref{eqn:cpncd} give
\begin{align*}
	C_NC_DQ \frac{k}{2(n-k)} & = \frac{\sqrt{\frac{k}{n-k}}}{(1-\sin^{2n}r)\sqrt{\left(1+\frac{k}{n-k}\sin^{2n}r\right)\left(\frac{k}{n-k}+\sin^{2n}r\right)}} \\
	& = \frac{1}{(1-\sin^{2n}r)\sqrt{\left(1+\sin^{2n}r\right)^2 + \frac{(n-2k)^2}{k(n-k)}\sin^{2n}r}}.
\end{align*}

This allows \eqref{eqn:cpntheta1} to be simplified as
\begin{equation}\label{eqn:cpntheta2}
	\cos\theta_r^{2k} = \frac{1-\sin^{2n}r}{\sqrt{\left(1+\sin^{2n}r\right)^2 + \frac{(n-2k)^2}{k(n-k)}\sin^{2n}r}},
\end{equation}
completing the proof of \thmref{thm:cpn}.


\subsection{Generalization to other bundles} 
\label{sec:generalization_to_other_bundles}
The manifold $M_r$ constructed above is topologically a $D^2$-bundle over $\CP^{n-1}$ with Euler class~1.  Other non-trivial $D^2$-bundles over $\CP^{n-1}$ have the same rational cohomology, so they also have a single Poincar\'e duality angle between the concrete realizations of the absolute and relative cohomology groups.

The $D^2$-bundle with Euler class $m$ over $\CP^{n-1}$ has boundary $L(m, 1)$, the lens space which is the quotient of $S^{2n-1}$ by the action of $\mathbb{Z}/m \mathbb{Z}$ given by
\[
	e^{2\pi i /m}\cdot (z_0, \ldots , z_n) = \left(e^{2\pi i/ m}z_0, \ldots , e^{2\pi i/ m}z_n\right).
\]
The lens space $L(m,1)$ fibers over $\CP^{n-1}$; when $n=2$ this corresponds to the non-singular Seifert fibration of the three-dimensional lens space $L(m,1)$ over $S^2$.

Let $L(m,1)_{t}$ be the quotient of $S^{2n-1}(\cos t)_{\sin t}$ by this action of $\mathbb{Z}/m \mathbb{Z}$.  Then $L(m,1)_t$ is a ``Berger lens space'' obtained from the standard lens space $L(m,1)$ by scaling the whole space by $\cos t$ and the circle fibers over $\CP^{n-1}$  by an additional factor of $\sin t$.  

Define the $2n$-dimensional Riemannian manifold $M_{r,m}$ modeled on  $M_r$ by replacing every appearance of $S^{2n-1}(\cos t)_{\sin t}$ in the geometric description of $M_r$ by $L(m,1)_t$.  Then $M_{r,m}$ is topologically a $D^2$-bundle over $\CP^{n-1}$ with Euler class $m$.

All of the computations in Sections~\ref{sec:finding_harmonic_forms}--\ref{sec:cpn_computing_the_poincar'e_duality_angle} follow through verbatim except that each appearance of $\text{vol}\, S^{2n-1}(1)$ must be replaced by $$\text{vol}\, L(m,1) = \frac{1}{m}\,\text{vol}\, S^{2n-1}(1).$$  The $m$'s cancel in the end, so it follows that

\begin{theorem}\label{thm:mrq}
	For $1 \leq k \leq n-1$, the Poincar\'e duality angle $\theta_{r,m}^{2k}$ between the concrete realizations of $H^{2k}(M_{r,m}; \mathbb{R})$ and $H^{2k}(M_{r,m}, \partial M_{r,m}; \mathbb{R})$ is given by
	\[
		\cos \theta_{r,m}^{2k} = \frac{(1-\sin^{2n}r)}{\sqrt{\left(1+\sin^{2n}r\right)^2 + \frac{(n-2k)^2}{k(n-k)}\sin^{2n}r}}.
	\]
\end{theorem}

Thus, even though the $M_{r,m}$ cannot be closed up by capping off with a ball, their Poincar\'e duality angles have the same asymptotic behavior as did the Poincar\'e duality angles for $M_r$.



\section{Poincar\'e duality angles on Grassmannians} 
\label{chap:pdangles_g2rn}

The example of $\CP^n - B_r(x)$ is a special case of a more general phenomenon in a couple of different ways.  $\CP^n$ is a simple example of a Grassmannian (namely, of complex lines in $\mathbb{C}^{n+1}$); it would be interesting to determine the Poincar\'e duality angles of other Grassmannians.  Also, removing a ball centered at some point in a closed manifold to get a manifold with boundary is the simplest way of removing a tubular neighborhood of a submanifold from a closed manifold.

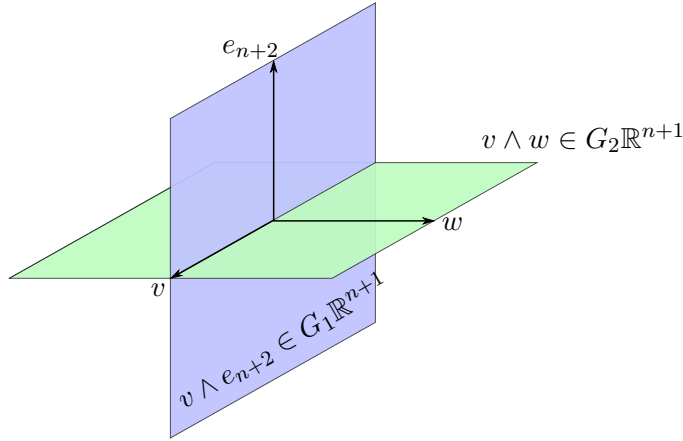
\begin{figure}[htbp]
	\centering
		\includegraphics[scale=.7]{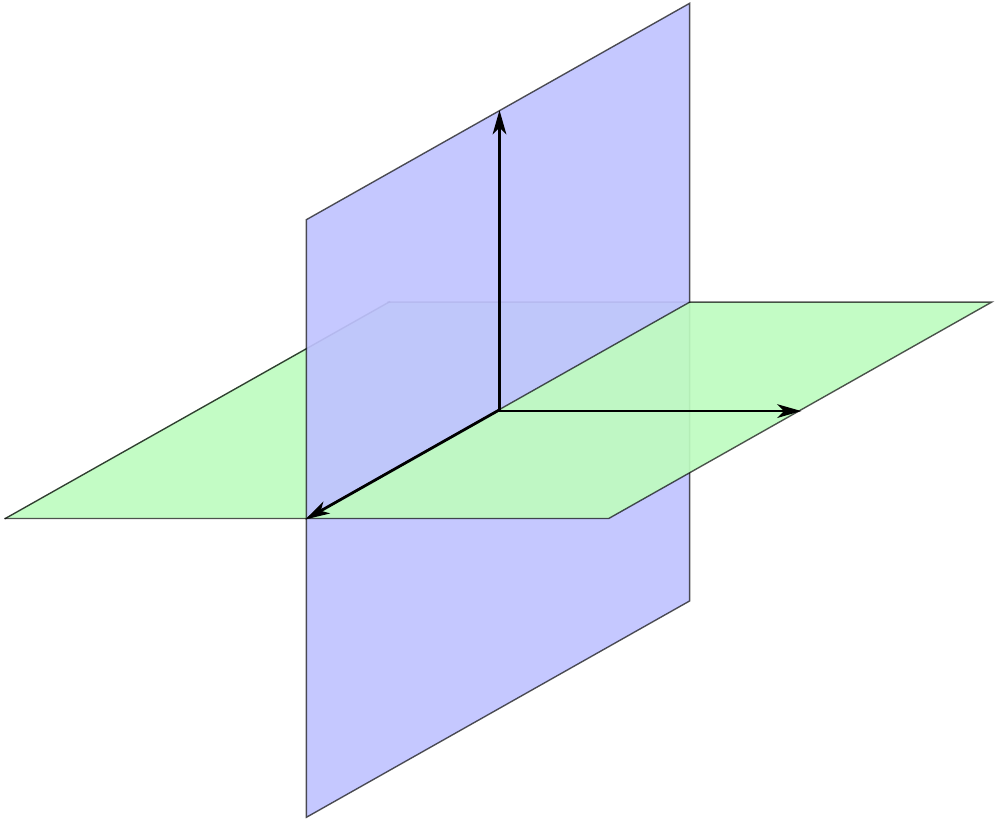}
		\put(-120,147){$e_{n+2}$}
		\put(-147,55){$v$}
		\put(-37,80){$w$}
		\put(-22,111){$v \wedge w \in G_2 \mathbb{R}^{n+1}$}
		\put(-139,12){\rotatebox{30}{$v \wedge e_{n+2} \in G_1 \mathbb{R}^{n+1}$}}
	\caption{Elements of the subGrassmannians $G_1 \mathbb{R}^{n+1}$ and $G_2 \mathbb{R}^{n+1}$}
	\label{fig:g2rnpoints}
\end{figure}

With that in mind, consider $G_2 \mathbb{R}^{n+2}$, the Grassmannian of oriented $2$-planes in $\mathbb{R}^{n+2}$.  The usual metric on $G_2 \mathbb{R}^{n+2} = SO(n+2)/SO(n)$ is the Riemannian submersion metric induced by the bi-invariant metrics on $SO(n+2)$ and $SO(n)$.  For any $m$ the bi-invariant metric on $SO(m)$ is inherited from the Euclidean metric on the space of $m \times m$ matrices and has the unfortunate feature that the standard basis vectors $E_{ij}$ for the Lie algebra $\mfr{so}(m)$ have length $\sqrt{2}$.  For the purposes of this section it is more convenient to reduce all linear dimensions in all the orthogonal groups by $\sqrt{2}$ and then assign any associated homogeneous spaces, such as $G_2 \mathbb{R}^{n+2}$, the resulting Riemannian submersion metric.

The Grassmannian $G_2 \mathbb{R}^{n+2}$ contains two obvious Grassmannian submanifolds,  $G_2 \mathbb{R}^{n+1}$ and $G_1 \mathbb{R}^{n+1}$. Oriented 2-planes in $\mathbb{R}^{n+1}$ are certainly also oriented 2-planes in $\mathbb{R}^{n+2}$, while an oriented line in $\mathbb{R}^{n+1}$ determines the 2-plane in $\mathbb{R}^{n+2}$ containing that line and the $x_{n+2}$-axis, oriented from the line to the $x_{n+2}$-axis (see Figure~\ref{fig:g2rnpoints}).  This is entirely analogous to the situation in $\CP^n$ described in Section~\ref{chap:pdangles_cpn}.  Letting $x = (0:0:\ldots :0:1)$, the relevant submanifolds there were $\CP^{n-1}$, the space of complex lines in $\mathbb{C}^n$, and $\{(0:0:\ldots :0:1)\}$, which can be identified with the space of complex 0-planes in $\mathbb{C}^n$.

Paralleling the $\CP^n$ story, define the one-parameter family of manifolds
\[
	N_r := G_2 \mathbb{R}^{n+2} - \nu_r\left(G_1 \mathbb{R}^{n+1}\right),
\]
where $\nu_r\left(G_1 \mathbb{R}^{n+1}\right)$ is the open tubular neighborhood of radius $r$ around $G_1 \mathbb{R}^{n+1}$.  Since an oriented line in $\mathbb{R}^{n+1}$ is determined by a unit vector in $\mathbb{R}^{n+1}$, the manifold $G_1 \mathbb{R}^{n+1}$ is just the unit sphere $S^n$.  Just as $\CP^{n-1}$ was the cut locus of $\{p\}$, $G_2 \mathbb{R}^{n+1}$ is the focal locus of $G_1 \mathbb{R}^{n+1}$ (and vice versa), so $\nu_r\left(G_1 \mathbb{R}^{n+1}\right)$ is topologically the tangent bundle of $G_1 \mathbb{R}^{n+1} \simeq S^n$.  Hence, the boundary 
\[
	\partial \left(\overline{\nu_r\left(G_1 \mathbb{R}^{n+1}\right)}\right) = \partial N_r
\] 
is homeomorphic to the unit tangent bundle $US^n$.  

Since $G_1 \mathbb{R}^{n+1}$ is the focal locus of $G_2 \mathbb{R}^{n+1}$, the manifold $N_r$ is a disk bundle over $G_2 \mathbb{R}^{n+1}$.  Both $G_2 \mathbb{R}^{n+2}$ and $N_r$ are $2n$-manifolds while $G_2 \mathbb{R}^{n+1}$ is of dimension $2(n-1) = 2n-2$, so $N_r$ is a $D^2$-bundle over $G_2 \mathbb{R}^{n+1}$.

The fibers over $G_2 \mathbb{R}^{n+1}$ are contractible, so $N_r$ has the same absolute cohomology as $G_2 \mathbb{R}^{n+1}$:
\[
	H^i(N_r; \mathbb{R}) = \begin{cases} \mathbb{R} & \text{for } i = 2k, \ 0 \leq k \leq n-1, \ k \neq \frac{n-1}{2} \\
	\mathbb{R}^2 & \text{if } n-1 \text{ is even and } i = n-1 \\
	0 & \text{otherwise}. \end{cases}
\]
Since $H^i(N_r; \mathbb{R}) \cong H^{2n-i}(N_r, \partial N_r ; \mathbb{R})$ by Poincar\'e--Lefschetz duality, the relative cohomology of $N_r$ can be deduced from the above.  Thus, the absolute and relative cohomology groups are both non-trivial in dimension $2k$ for $1 \leq k \leq n-1$.  When $n$ is even, $\partial N_r \simeq US^n$ has the same rational cohomology as $S^{2n-1}$, so all of this cohomology is interior.  When $n$ is odd, $US^n$ has the same rational cohomology as $S^n \times S^{n-1}$, so  all of this cohomology is interior except that $H^{n-1}(N_r; \mathbb{R})$ and $H^{n-1}(M, \partial M; \mathbb{R})$ could each potentially have a 1-dimensional boundary subspace.  This turns out to be the case and the goal of this section is to prove \thmref{thm:g2rn}:

\begin{theoremg2rn}
	For $1 \leq k \leq n-1$ there is exactly one Poincar\'e duality angle $\theta_r^{2k}$ between the concrete realizations of $H^{2k}(N_r; \mathbb{R})$ and $H^{2k}(N_r, \partial N_r; \mathbb{R})$ given by
	\[
		\cos \theta_r^{2k} = \frac{1-\sin^nr}{\sqrt{(1+\sin^n r)^2 + \frac{(n-2k)^2}{k(n-k)}\sin^nr}}.
	\]
\end{theoremg2rn}

As in the $\CP^n$ story, $\cos \theta_r^{2k} = 1 + O(r^n)$ for small $r$.  The number $n$ gives the codimension of $ G_1 \mathbb{R}^{n+1}$, so the Poincar\'e duality angles not only go to zero as $r \to 0$, but they apparently detect the codimension of the manifold removed.  Also, for $r$ close to its maximum value of $\pi/2$, the quantity $\cos \theta_r^{2k} = O(r^2)$ and $2 = \text{codim}\, G_2 \mathbb{R}^{n+1}$, which is the manifold onto which $N_r$ collapses as $r \to \pi/2$.

\subsection{The geometry of $N_r$} 
\label{sec:g2rngeometry}
The Grassmannian $G_2 \mathbb{R}^{n+2}$ can be viewed as a cohomogeneity one manifold as follows.  As a homogeneous space, 
\[
	G_2 \mathbb{R}^{n+2} = \frac{SO(n+2)}{SO(2) \times SO(n)},
\]
where the $SO(2)$ acts by rotating the $x_1x_2$-plane and the $SO(n)$ rotates the orthogonal $n$-plane.

The group $SO(n+1)$ embeds into $SO(n+2)$ as the subgroup fixing the $x_{n+2}$-axis.  This embedding induces an action of $SO(n+1)$ on $G_2 \mathbb{R}^{n+2}$. This induced action of $SO(n+1)$ is transitive on the copies of $G_2 \mathbb{R}^{n+1}$ and $G_1 \mathbb{R}^{n+1}$ sitting inside of $G_2 \mathbb{R}^{n+2}$, and of course
\[
	G_2 \mathbb{R}^{n+1} = \frac{SO(n+1)}{SO(2) \times SO(n-1)}, \quad G_1 \mathbb{R}^{n+1} = S^n = \frac{SO(n+1)}{SO(n)}.
\]

The action of $SO(n+1)$ preserves the principal angle a 2-plane makes with the $x_{n+2}$-axis, but is transitive on each collection of 2-planes making the same such  angle, so the principal orbits of this action are those 2-planes making an acute angle $\phi$ with the $x_{n+2}$-axis.  The principal orbit corresponding to the angle $\phi$ forms the hypersurface at constant distance $\phi$ from $G_1 \mathbb{R}^{n+1}$ and distance $\pi/2-\phi$ from $G_2 \mathbb{R}^{n+1}$.

The set of oriented 2-planes making an acute angle $\phi$ with the $x_{n+2}$-axis is homeomorphic to the Stiefel manifold $V_2 \mathbb{R}^{n+1}$ of oriented 2-frames in $\mathbb{R}^{n+1}$.  This can be seen as follows: if $P$ is an oriented 2-plane making a principal angle $\phi$ with the $x_{n+2}$-axis, let $v$ be the unit vector in $P$ which realizes this angle.  Let $w$ be the orthogonal unit vector in $P$ such that $(v,w)$ gives the orientation of $P$.  Then $w$ is orthogonal to the $x_{n+2}$-axis, so $w \in \mathbb{R}^{n+1}$.  Since $\phi \neq 0$, the orthogonal projection $\mathsf{proj}\, v$ of $v$ to $\mathbb{R}^{n+1}$ has length $\sin \phi \neq 0$, so $\left(\frac{\mathsf{proj}\, v}{\sin \phi}, w\right)$ gives an oriented 2-frame in $\mathbb{R}^{n+1}$. 
 
The upshot is that $G_2 \mathbb{R}^{n+2}$ is the cohomogeneity one manifold corresponding to the group diagram
\begin{equation*}
	\hspace{-.8in}\begin{diagram}[size=2em]
		& &  & SO(n+1) & &  \\
		&& \ldLine & &   \rdLine &\\
		SO(2) \times &SO(n-1) & & & & SO(n) \\
		& &\rdLine & & \ldLine &  \\
		 &  && SO(n-1) &  &  \\
	\end{diagram}
\end{equation*}
with principal orbits
\[
	\frac{SO(n+1)}{SO(n-1)} = V_2 \mathbb{R}^{n+1} \simeq US^n.
\]

Each principal orbit can be viewed as a bundle over $G_2 \mathbb{R}^{n+1}$ with fiber
\[
	\frac{SO(2)\times SO(n-1)}{SO(n-1)} = SO(2) \simeq S^1.
\]

\begin{figure}[htbp]
	\centering
		\includegraphics[scale=1.3]{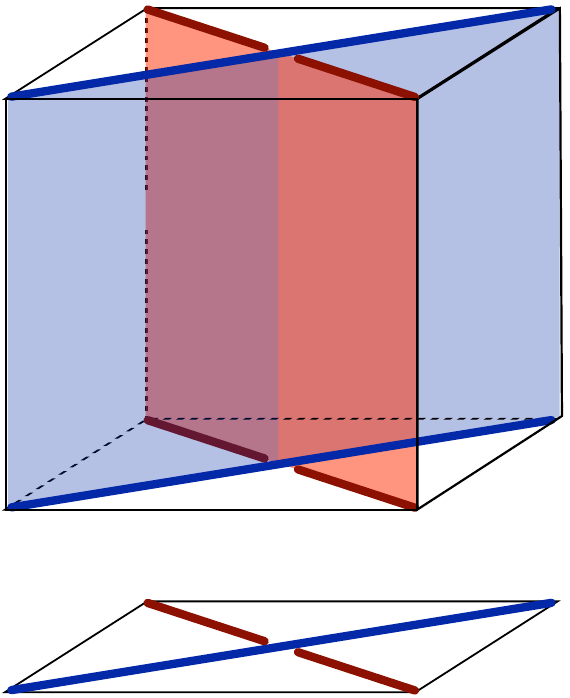}
		\put(-260,18){$G_2 \mathbb{R}^{n+1}$}
		\put(-260,137){$V_2 \mathbb{R}^{n+1}$}
		\put(-172,13){\rotatebox{10}{$G_2 \mathbb{R}^n$}}
		\put(-80,15){\rotatebox{-20}{$G_1 \mathbb{R}^n$}}
		\put(-106,165){\rotatebox{-19}{$\simeq S^n \times S^1$}}
		\put(-108,185){\rotatebox{-19}{$G_1 \mathbb{R}^n \times S^1$}}
		\put(-172,178){\rotatebox{10}{$V_2 \mathbb{R}^n$}}
	\caption{The Stiefel bundle $V_2\mathbb{R}^{n+1} \to G_2 \mathbb{R}^{n+1}$}
	\label{fig:stiefelbundle}
\end{figure}
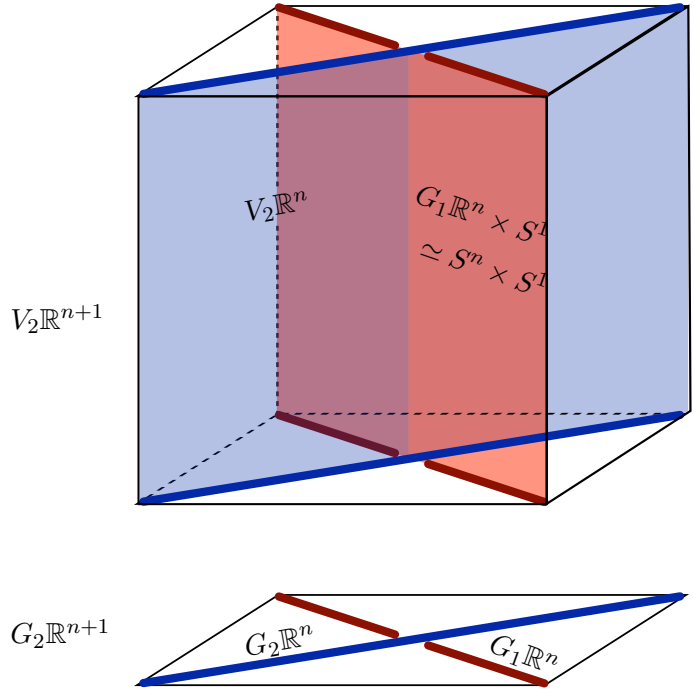

This bundle is topologically just the Stiefel bundle over the Grassmannian $G_2 \mathbb{R}^{n+1}$, which is pictured in Figure~\ref{fig:stiefelbundle} (for a detailed description of the Stiefel bundle in the case $n=3$, cf.\ \cite{GluckZiller}).  The base manifold $G_2 \mathbb{R}^{n+1}$ contains as subGrassmannians $G_2 \mathbb{R}^n$ and $G_1 \mathbb{R}^n$, where $G_1 \mathbb{R}^n$ consists of the oriented 2-planes in $\mathbb{R}^{n+1}$ containing the $x_{n+1}$ axis.  The fact that such 2-planes contain a distinguished vector implies that the bundle over $G_1 \mathbb{R}^n$ is trivial.  On the other hand, the bundle over $G_2 \mathbb{R}^n$ is the Stiefel bundle $V_2 \mathbb{R}^n \to G_2 \mathbb{R}^n$, which is certainly non-trivial.

When $n-1$ is even, $H^{n-1}(G_2 \mathbb{R}^{n+1}; \mathbb{R}) \cong H^{n-1}(N_r; \mathbb{R})$ is 2-dimensional and is generated by cohomology classes dual to $G_2 \mathbb{R}^{\frac{n+1}{2}} \subset G_2 \mathbb{R}^n$ and $G_1 \mathbb{R}^n$ (for example, $G_2 \mathbb{R}^{\frac{n+1}{2}}$ is calibrated by the appropriate power of the K\"ahler form \cite{Wirtinger, Federer}).  Since the bundle over $G_1 \mathbb{R}^n$ is trivial, the cycle $G_1 \mathbb{R}^n$ can be pushed out to the boundary, meaning that the dual cohomology class comes from the boundary subspace.  Thus, even though $H^{n-1}(N_r; \mathbb{R})$ is two-dimensional when $n-1$ is even, the interior subspace is only one-dimensional, so there is a single Poincar\'e duality angle in dimension $n-1$.

Each principal orbit in $G_2 \mathbb{R}^{n+2}$ can also be viewed as a bundle over $G_1 \mathbb{R}^{n+1}$ with fiber 
\[
	\frac{SO(n)}{SO(n-1)} = S^{n-1},
\]
corresponding to the fact that $V_2 \mathbb{R}^{n+1}$ is topologically the unit tangent bundle of $S^n$.

Let $\left(V_2 \mathbb{R}^{n+1}\right)_t$ denote the principal orbit at distance $t$ from $G_2 \mathbb{R}^{n+1}$ (i.e. the collection of 2-planes making an angle $\pi/2-t$ with the $x_{n+2}$-axis).  The metric on $\left(V_2 \mathbb{R}^{n+1}\right)_t$ is of course not the homogeneous metric:  the circle fibers over $G_2 \mathbb{R}^{n+1}$ and the $S^{n-1}$ fibers over $G_1 \mathbb{R}^{n+1}$ are separately scaled by some factor depending on $t$.

\begin{figure}[htbp]
		\hspace{-80px}\includegraphics[scale=.7]{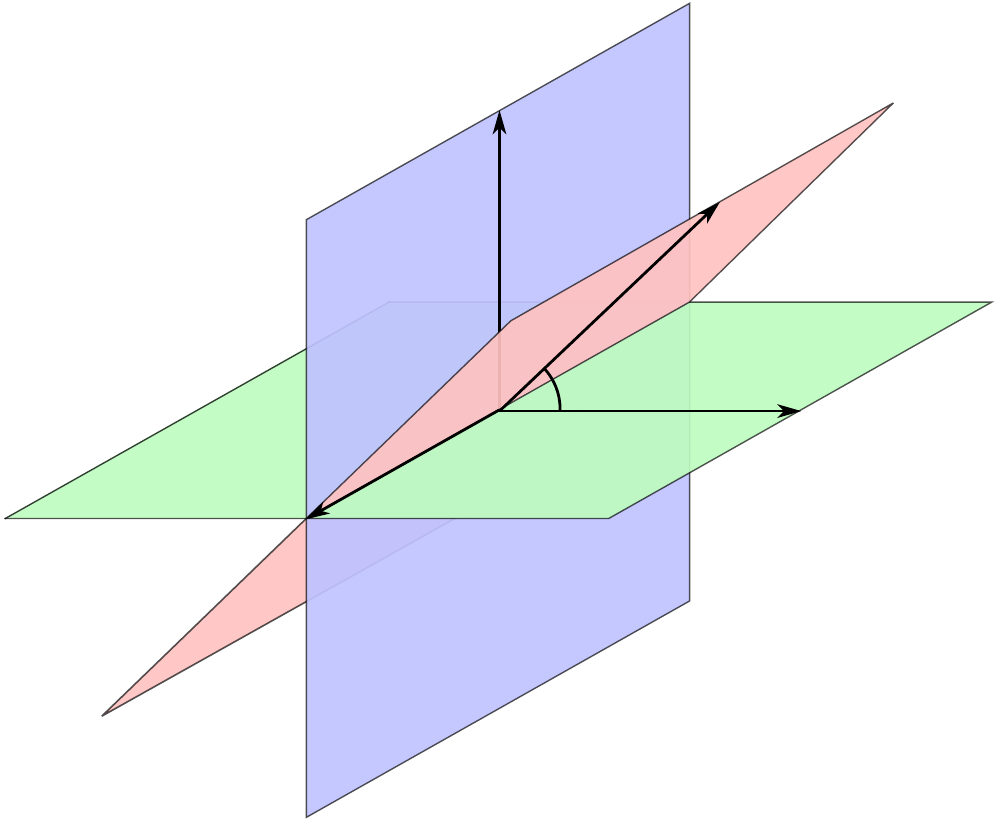}
		\put(-120,147){$e_{n+2}$}
		\put(-152,55.5){$v$}
		\put(-37,80){$w$}
		\put(-32,150){$\gamma(t) = v \wedge (\cos t \ w + \sin t \ e_{n+2})$}
		\put(-43,97){$v \wedge w $}
		\put(-117,25){\rotatebox{30}{$v \wedge e_{n+2}$}}
		\put(-85,86){$t$}
	\caption{A point on a geodesic from $G_2 \mathbb{R}^{n+1}$ to $G_1 \mathbb{R}^{n+1}$}
	\label{fig:g2rngeodesic}
\end{figure}
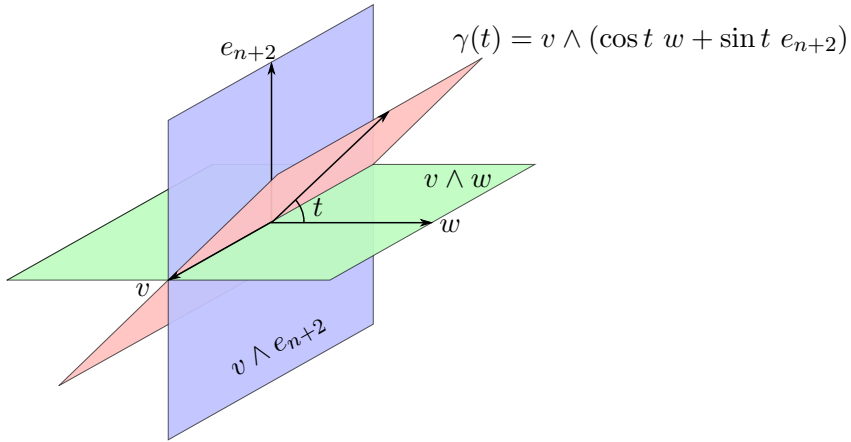

To determine the scaling on the fibers, note that the geodesics from $G_2 \mathbb{R}^{n+1}$ to $G_1 \mathbb{R}^{n+1}$ are given by picking a unit vector in the 2-plane and rotating it towards the $x_{n+2}$ axis, as illustrated in Figure~\ref{fig:g2rngeodesic}.  In particular, if $v \wedge w$ denotes the 2-plane spanned by the unit vectors $v, w \in \mathbb{R}^{n+1}$, then a geodesic from $v \wedge w \in G_2 \mathbb{R}^{n+1}$ to $G_1 \mathbb{R}^{n+1}$ is given by
\begin{equation}\label{eqn:g2rngeodesics}
	\gamma(t) = v \wedge (\cos t\ w + \sin t\ e_{n+2}) = \cos t \ v \wedge w + \sin t \ v \wedge e_{n+2},
\end{equation}
where $\{e_1, \ldots , e_{n+2}\}$ is the standard basis for $\mathbb{R}^{n+2}$.  

For $\theta \in [0,2\pi]$,
\[
	(\cos \theta\ v + \sin \theta\ w) \wedge (\cos \theta\ w - \sin \theta\ v)
\]
determines the same 2-plane as $v \wedge w$, but the corresponding geodesic is different:
\begin{align*}
	\gamma_\theta(t) & = (\cos \theta\ v + \sin \theta\ w) \wedge (\cos t\, (\cos \theta\ w - \sin \theta\ v) + \sin t\ e_{n+2})\\
	& = \cos t \, v \wedge w + \sin t\, (\cos \theta\, w - \sin \theta\, v) \wedge e_{n+2}.
\end{align*}
As $\theta$ varies from 0 to $2\pi$, $\gamma_\theta(v)$ sweeps out the circle fiber over the point $v \wedge w \in G_2 \mathbb{R}^{n+1}$.  

To determine the length of this circle, compute
\[
	\frac{\partial \gamma_\theta(t)}{\partial \theta} = \sin t\,(-\sin \theta\ w - \cos \theta\ v) \wedge e_{n+2}.
\]
Hence,
\[
	\left\|\frac{\partial \gamma_\theta(t)}{\partial \theta} \right\| = |\sin t\,|\, \|(-\sin \theta\ w - \cos \theta\ v) \wedge e_{n+2}\| = |\sin t\,|
\]
since both $\sin \theta\, w - \cos \theta\, v$ and $e_{n+2}$ are unit vectors.  The parameter $t$ is strictly between $0$ and $\pi/2$, so $|\sin t\,| = \sin t$ and the length of the circle fiber is given by
\[
	\int_0^{2\pi} \left\|\frac{\partial \gamma_\theta(t)}{\partial \theta} \right\|d\theta = \int_0^{2\pi} \sin t\ d\theta = 2\pi \sin t.
\]
In other words, the circle fiber over each point in $G_2 \mathbb{R}^{n+1}$ is a circle of radius $\sin t$.  

On the other hand, if $v \in G_1 \mathbb{R}^{n+1} = S^n$, then $v$ corresponds to the 2-plane $v \wedge e_{n+2}$ in the concrete copy of $G_1 \mathbb{R}^{n+1}$ sitting inside $G_2 \mathbb{R}^{n+2}$.  Using the expression \eqref{eqn:g2rngeodesics} for the geodesics from $G_2 \mathbb{R}^{n+1}$ to $G_1 \mathbb{R}^{n+1}$, the fiber over $v$ in $\left(V_2 \mathbb{R}^{n+1}\right)_t$ consists of those 2-planes
\[
	\cos t\ v \wedge w + \sin t\ v \wedge e_{n+2},
\]
where $w$ is a unit vector in $\mathbb{R}^{n+1}$ orthogonal to $v$.  The collection of unit vectors orthogonal to $v$ forms an $(n-1)$-sphere, confirming that the fiber over $v$ is a copy of $S^{n-1}$.  

If $w_1$ and $w_2$ are unit vectors orthogonal both to $v$ and to each other, then
\[
	\sigma(\theta, t) = \cos t\ v \wedge (\cos \theta\ w_1 + \sin \theta\ w_2) + \sin t\ v \wedge e_{n+2}
\]
determines a great circle in the fiber over $v \in G_1 \mathbb{R}^{n+1}$.  The length of the tangent vector to this circle is
\[
	\left\|\frac{\partial \sigma(\theta, t)}{\partial \theta} \right\| = \left\|\cos t\ v \wedge (-\sin \theta\ w_1 + \cos \theta\ w_2) \right\| = |\cos t\,| = \cos t,
\]
since $0 < t < \pi/2$.  Therefore, the great circle determined by $\sigma$ has length
\[
	\int_{0}^{2\pi} \left\|\frac{\partial \sigma(\theta, t)}{\partial \theta} \right\| d\theta = \int_0^{2\pi} \cos t\ d\theta = 2\pi \cos t.
\]
Since the same scaling holds for any great circle in the fiber over $v$, this fiber is a round $S^{n-1}$ of radius $\cos t$.

Combining all of the above, the circle fibers of $\left(V_2 \mathbb{R}^{n+1}\right)_t$ over $G_2 \mathbb{R}^{n+1}$ are scaled by $\sin t$ while the $S^{n-1}$ fibers over $G_1 \mathbb{R}^{n+1}$ are scaled by $\cos t$.  For $0 < t < \pi/2 - r$,  $\left(V_2 \mathbb{R}^{n+1}\right)_t$ is precisely the hypersurface at constant distance $t$ from $G_2 \mathbb{R}^{n+1}$ inside the manifold $N_r = G_2 \mathbb{R}^{n+2} - \nu_r\left(G_1 \mathbb{R}^{n+1}\right)$, so this gives a fairly complete geometric picture of $N_r$.

\subsection{Invariant forms on $N_r$} 
\label{sec:invariant_forms_on_m_r_}
The goal is to find Neumann and Dirichlet harmonic fields on $N_r$ and to compute the angles between them; these angles will then be the Poincar\'e duality angles of $N_r$.  Such harmonic fields must be isometry-invariant, so the question arises: what forms on $N_r$ are isometry-invariant?  Since the $t$ direction is invariant under isometries of $N_r$ and since the isometry group of $N_r$ is $SO(n+1)$, this is equivalent to asking what forms on the hypersurfaces $\left(V_2 \mathbb{R}^{n+1}\right)_t$ are invariant under the action of $SO(n+1)$.

Let $\{e_1, \ldots , e_{n+2}\}$ be the standard orthonormal basis for $\mathbb{R}^{n+2}$ and identify $\left(V_2 \mathbb{R}^{n+1}\right)_t$ with the standard $V_2 \mathbb{R}^{n+1}$ via the diffeomorphism
\[
	\cos t\, v \wedge w + \sin t\, v \wedge e_{n+2} \mapsto (v,w).
\]
Moreover, consider $V_2 \mathbb{R}^{n+1}$ in the guise of the unit tangent bundle of $S^n$.  Under this identification, $(e_1, e_2) \in US^n$ sits in the fiber over the point $e_1 \wedge e_2 \in G_2 \mathbb{R}^{n+1}$.  

The tangent bundle $US^n$ can be seen concretely as the space
\[
	\{(x,y) \in \mathbb{R}^{n+1} \times \mathbb{R}^{n+1} : \|x\| = \|y\| = 1, \langle x, y\rangle = 0\},
\]
with the caveat that this interpretation is not quite right metrically (see below).  In this view, the tangent directions to a point $(x, y)$ are those vectors $(v,w) \in \mathbb{R}^{n+1} \times \mathbb{R}^{n+1}$ such that
\[
	\langle x,v\rangle = 0, \quad \langle y, w \rangle = 0, \quad \langle x, w \rangle + \langle y, v \rangle = 0.
\]
Therefore, a basis for the tangent space at $(e_1, e_2) \in US^n$ is 
\[
	(e_2, -e_1), (e_3,0), \ldots , (e_{n+1}, 0), (0, e_3), \ldots , (0, e_{n+1}).
\]
Although $(e_2, -e_1)$ is not a unit vector in $\mathbb{R}^{n+1} \times \mathbb{R}^{n+1}$, it corresponds to the geodesic flow direction in $US^n$ and thus {\it is} a unit vector there.  Hence, the above gives an orthonormal basis for the tangent space to $(e_1, e_2) \in US^n$.

The fiber over $(e_1,e_2)\in G_2 \mathbb{R}^{n+1}$ corresponds to rotating the $e_1e_2$-plane from $e_1$ towards $e_2$ and leaving the other directions fixed, so the fiber direction is given by $(e_2, -e_1)$.  The fiber over $G_1 \mathbb{R}^{n+1}$ consists of those directions which fix the first factor, so the fiber directions are $(0, e_3), \ldots , (0, e_{n+1})$.

The isotropy subgroup of $(e_1, e_2)$ is the $SO(n-1)$ which rotates the $e_3\cdots e_{n+1}$-plane and fixes the $e_1e_2$-plane.  Since $US^n = \frac{SO(n+1)}{SO(n-1)}$ is a homogeneous space, the invariant forms on $US^n$ pull back to left-invariant forms on $SO(n+1)$ which are invariant under conjugation by $SO(n-1)$.

Let $\{E_{ij}| 1\leq i < j \leq n+1\}$ be the standard basis for the Lie algebra $\mfr{so}(n+1)$, where $E_{ij}$ is the direction corresponding to rotating the $e_ie_j$-plane from $e_i$ towards $e_j$ and leaving everything else fixed.  Letting $E_{ji} = -E_{ij}$, the Lie brackets are computed by
\[
	[E_{ij}, E_{\ell m}] = E_{ij}E_{\ell m} - E_{\ell m}E_{ij} = -\delta_{j\ell}E_{im}
\]
for $i \neq m$.  

Let $V_{ij}$ be the left-invariant vector field on $SO(n+1)$ determined by $E_{ij}$ and let $\Phi_{ij}$ be the dual left-invariant 1-form.  Exterior derivatives of the $\Phi_{ij}$ can be computed using Cartan's formula:
\[
	d\Phi_{ij}(V,W) = V\Phi_{ij}(W) - W \Phi_{ij}(V) - \Phi_{ij}([V,W]).
\]
For example,
\[
	d\Phi_{ij}(E_{i\ell}, E_{\ell j}) = -\Phi_{ij}([E_{i\ell}, E_{\ell j}]) = -\Phi_{ij}(-E_{ij}) = 1,
\]
so
\[
	d\Phi_{ij} = \sum_{\ell \neq i, j} \Phi_{i\ell}\wedge \Phi_{\ell j}.
\]
If $\Phi_{ij}$ is invariant under conjugation by $SO(n-1)$, then
\[
	\Phi_{ij} = i^*\varphi_{ij},
\]
where $\varphi_{ij}$ is an invariant form on $US^n$.  

The vector $E_{12}$ corresponds to the fiber direction at $(e_1, e_2)$.  Since the group $SO(n-1)$ fixes the $e_1e_2$-plane here, $\Phi_{12}$ is invariant under conjugation by $SO(n-1)$ and  so $\Phi_{12} = i^*\varphi_{12}$ where, under the identification of $US^n$ with $\left(V_2 \mathbb{R}^{n+1}\right)_t$, the form $\varphi_{12}$ is the 1-form on $\left(V_2 \mathbb{R}^{n+1}\right)_t$ dual to  the fiber direction.  Moreover, since exterior differentiation preserves invariance, $d\varphi_{12}$ is also an invariant form and
\[
	i^*d \varphi_{12} = d\,i^*\varphi_{12} = d\Phi_{12} = \sum_{k=3}^{n+1} \Phi_{1k} \wedge \Phi_{k2}.
\]
Therefore,
\begin{equation}\label{eqn:dphi12}
	d \varphi_{12} = \sum_{\ell=3}^{n+1} \varphi_{1\ell}\wedge \varphi_{\ell 2},
\end{equation}
where $\varphi_{1\ell}$ is dual to the $(e_\ell, 0)$ direction and $\varphi_{2\ell}$ is dual to the $(0, e_\ell)$ direction.  In particular, this means the forms $\varphi_{2\ell}$ are dual to the fiber directions over $G_1 \mathbb{R}^{n+1} = S^n$.  

The volume form on $US^n$ is given by 
\begin{equation}\label{eqn:usnvolform}
	\text{dvol}_{US^n} = \varphi_{12} \wedge \varphi_{13} \wedge \ldots \wedge \varphi_{1(n+1)} \wedge \varphi_{23} \wedge \ldots \wedge \varphi_{2(n+1)},
\end{equation}
so, using \eqref{eqn:dphi12},
\[
	\varphi_{12} \wedge (d \varphi_{12})^{n-1} = (n-1)!\, \text{dvol}_{US^n}.\footnote{The fact that $\varphi_{12} \wedge (d \varphi_{12})^{n-1}$ is nowhere vanishing also follows from the fact that $\varphi_{12}$ is dual to the geodesic flow direction and, therefore, is the canonical contact form on $US^n$.}
\]

Since $d\varphi_{12}$ is invariant, the exterior powers of $d \varphi_{12}$ are also invariant.  Thus, 
\[
	 (d \varphi_{12})^k \quad \text{and} \quad \varphi_{12} \wedge (d \varphi_{12})^{k-1}
\]
are invariant $2k$- and $(2k-1)$-forms, respectively, on $\left(V_2 \mathbb{R}^{n+1}\right)_t$.

Therefore, since the desire is to find invariant forms on $N_r$, the obvious candidates are forms like 
\begin{equation}\label{eqn:g2rnomega}
	 \omega = f(t) (d \alpha)^k + g(t)\, \alpha \wedge (d \alpha)^{k-1} \wedge \tau
\end{equation}
where $\tau$ is dual to the unit vector in the $t$ direction and $\alpha$ denotes the 1-form on $N_r$ which restricts to $\varphi_{12}$ on each $\left(V_2 \mathbb{R}^{n+1}\right)_t$.


\subsection{Finding harmonic fields on $N_r$} 
\label{sec:g2rnharmonicfields}
The goal of this section is to find harmonic fields $\omega_N \in \mathcal{H}^{2k}_N(N_r)$ and $\omega_D \in \mathcal{H}^{2k}_D(N_r)$.  Following the template given above in \eqref{eqn:g2rnomega}, let
\begin{align*}
	\omega_N & := f_N(t) (d\alpha)^k + g_N(t)\, \alpha \wedge (d\alpha)^{k-1} \wedge \tau \\
	\omega_D & := f_D(t) (d\alpha)^k + g_D(t)\, \alpha \wedge (d\alpha)^{k-1} \wedge \tau.
\end{align*}

Then
\[
	d\omega_N = \left[f_N'(t) + g_N(t)\right] (d\alpha)^k \wedge \tau,
\]
so $\omega_N$ is closed if and only if
\begin{equation}\label{eqn:g2rngn}
	g_N(t) = -f_N'(t).
\end{equation}
Likewise, $\omega_D$ is closed if and only if $g_D(t) = -f_D'(t)$.

To determine the co-closed condition, it is necessary to compute the Hodge star of both $(d\alpha)^k$ and $\alpha \wedge (d\alpha)^{k-1} \wedge \tau$.  

The form $\alpha$ is dual to the fiber direction over $G_2 \mathbb{R}^{n+1}$, which was shown in Section~\ref{sec:g2rngeometry} to be scaled by $\sin t$.  Therefore, the form dual to the unit vector in the fiber direction is $\sin t\, \alpha$.  

On the other hand, $d\alpha$ restricts to $\sum \varphi_{1\ell}\wedge \varphi_{\ell 2}$ on each $\left(V_2 \mathbb{R}^{n+1}\right)_t$ and the forms $\varphi_{2\ell} = -\varphi_{\ell 2}$ are dual to the fiber directions over $G_1 \mathbb{R}^{n+1}$.  Since the fiber over $G_1 \mathbb{R}^{n+1}$ is scaled by $\cos t$, the forms dual to the unit vectors in the fiber directions are $\cos t\, \varphi_{2\ell}$ for $3 \leq \ell \leq n+1$.

The other directions tangent to $\left(V_2 \mathbb{R}^{n+1}\right)_t$ do not experience any scaling as $t$ varies and of course $\tau$ is already dual to the unit vector in the $t$ direction.  Therefore, using \eqref{eqn:usnvolform}, the volume form on $N_r$ is given by
\[
	\frac{1}{(n-1)!}\sin t \cos^{n-1}t \, \alpha \wedge (d\alpha)^{n-1} \wedge \tau
\] 
and the relevant Hodge stars are 
\[
	\star \left( \frac{1}{(k-1)!} \sin t \cos^{k-1} t \, \alpha \wedge (d\alpha)^{k-1} \wedge \tau \right)  = \frac{1}{(n-k)!} \cos^{n-k}t \, (d\alpha)^{n-k} 
\]
and
\[
	\star \left( \frac{1}{k!} \cos^k t\, (d\alpha)^k\right) = \frac{1}{(n-k-1)!} \sin t \cos^{n-k-1} t\, \alpha \wedge (d\alpha)^{n-k-1} \wedge \tau.
\]

Thus,
\begin{align}
\nonumber	\star\, \omega_N & = \frac{(k-1)!}{(n-k-1)!} \left[ k \sin t \cos^{n-2k-1}t\, f_N(t)\, \alpha \wedge (d\alpha)^{n-k-1} \wedge \tau \vphantom{\frac{\cos^n}{\sin}}\right. \\
\label{eqn:g2rnstaromegan} & \qquad \qquad \qquad \quad 	\left. + \frac{1}{n-k} \frac{\cos^{n-2k+1}t}{\sin t}g_N(t)\, (d\alpha)^{n-k}\right],
\end{align}
and so
\begin{align*}
	d\star \omega_N & = \frac{(k-1)!}{(n-k-1)!}\left[k \sin t \cos^{n-2k-1}t\, f_N(t)\vphantom{\frac{k}{n-k}}\right. \\
	& \qquad \qquad \qquad \quad - \frac{1}{n-k}\left(\left((n-2k+1)\cos^{n-2k}t + \frac{\cos^{n-2k+2}t}{\sin^2 t}\right)g_N(t)\right.\\
	& \qquad \qquad \qquad \quad \left.\left. - \frac{\cos^{n-2k+1}t}{\sin t}g_N'(t)\right)\right] \alpha \wedge (d\alpha)^{n-k-1} \wedge \tau.
\end{align*}
Using the condition \eqref{eqn:g2rngn} that $g_N(t) = -f_N'(t)$ and simplifying yields the ODE
\[
	0 = f_N''(t) - \left((n-2k+1)\tan t + \cot t\right)f_N'(t) - k(n-k)\tan^2 t\, f_N(t).
\]
Solutions to this differential equation take the form
\[
	f_N(t) = C_1 \cos^k t + C_2 \frac{1}{\cos^{n-k}t},
\]
so 
\[
	g_N(t) = kC_1 \sin t \cos^{k-1}t - (n-k)C_2 \frac{\sin t}{\cos^{n-k+1}t}.
\]
The fact that $\omega_N$ satisfies the Neumann boundary condition is equivalent to requiring that
\[
	0 = g_N(\pi/2-r) = kC_1 \cos r \sin^{k-1}r - (n-k)C_2 \frac{\cos r}{\sin^{n-k+1}r},
\]
which means that
\[
	C_2 = \frac{k}{n-k} C_1 \sin^n r.
\]
Letting $C_N = C_1$ and substituting into the above expressions for $f_N(t)$ and $g_N(t)$ yields
\begin{align}
	\label{eqn:g2rnfngn} f_N(t) & = C_N \left[ \cos^k t + \frac{k}{n-k} \sin^n r\, \frac{1}{\cos^{n-k}t} \right] \\
	\nonumber g_N(t) & = C_N\left[ k \sin t \cos^{k-1}t - k \sin^n r\, \frac{\sin t}{\cos^{n-k+1}t}\right].
\end{align}

The functions $f_D(t)$ and $g_D(t)$ satisfy the same differential equation, but the Dirichlet condition is equivalent to $0 = f_D(\pi/2-r)$, which implies
\begin{align}
	\label{eqn:g2rnfdgd} f_D(t) & = C_D \left[ \cos^k t -\sin^n r\, \frac{1}{\cos^{n-k}t} \right] \\
	\nonumber g_D(t) & = C_D \left[ k \sin t \cos^{k-1}t + (n-k) \sin^n r\, \frac{\sin t}{\cos^{n-k+1}t} \right].
\end{align}


\subsection{Normalizing the forms} 
\label{sec:g2rnnormalizing_the_forms}
The goal of this section is to find the values of $C_N$ and $C_D$ which will ensure $\|\omega_N\|_{L^2} = 1 = \|\omega_D\|_{L^2}$.

Using the definition $\omega_N = f_N(t) (d\alpha)^k + g_N(t)\, \alpha \wedge (d\alpha)^{k-1} \wedge \tau$ and the expression \eqref{eqn:g2rnstaromegan} for the form $\star\, \omega_N$, 
\begin{align*}
	\omega_N \wedge \star\, \omega_N & = \frac{(k-1)!}{(n-k-1)!} \left[\vphantom{\frac{\cos^{n-2k+1}t}{\sin t}} k \sin t \cos^{n-2k-1}t\, f_N(t)^2 \right. \\
	& \qquad \qquad \quad \qquad \left. + \frac{1}{n-k} \frac{\cos^{n-2k+1}t}{\sin t}g_N(t)^2\right] \alpha \wedge (d\alpha)^{n-1} \wedge \tau.
\end{align*}

Therefore,
\begin{align*}
	\langle \omega_N , \omega_N \rangle_{L^2} & = \int_{N_r} \omega_N \wedge \star \, \omega_N \\
	& = \int_{N_r} \frac{(k-1)!}{(n-k-1)!} \left[k \sin t \cos^{n-2k-1}t\, f_N(t)^2 \vphantom{\frac{\cos^n}{\sin}}\right.\\
	& \qquad \qquad \qquad \qquad \quad \left. + \frac{1}{n-k} \frac{\cos^{n-2k+1}t}{\sin t}g_N(t)^2\right] \alpha \wedge (d\alpha)^{n-1} \wedge \tau.
\end{align*}

Integrating away the $\left(V_2 \mathbb{R}^{n+1}\right)_t$ yields
\[
	Q \int_0^{\pi/2-r}  \left[k \sin t \cos^{n-2k-1}t\, f_N(t)^2 + \frac{1}{n-k} \frac{\cos^{n-2k+1}t}{\sin t}g_N(t)^2\right]dt,
\]
where
\begin{equation}\label{eqn:g2rnq}
	Q := \frac{(n-1)!(k-1)!}{(n-k-1)!}\,\text{vol}\,V_2 \mathbb{R}^{n+1} .
\end{equation}
A computation taking into account the expressions in \eqref{eqn:g2rnfngn} for $f_N(t)$ and $g_N(t)$ shows that
\begin{align*}
	\langle \omega_N, \omega_N \rangle_{L^2} & = C_N^2 Q \int_0^{\pi/2-r} \left[\frac{kn}{n-k}\sin t \cos^{n-1}t + \frac{k^2n}{(n-k)^2}\sin^{2n}r\, \frac{\sin t}{\cos^{n+1}t}\right]dt \\
	& = C_N^2 Q\frac{k}{n-k}\left(1+\frac{2k-n}{n-k}\sin^{n}r - \frac{k}{n-k}\sin^{4n}r\right) \\
	& = C_N^2 Q \frac{k}{n-k}\left(1+\frac{k}{n-k}\sin^{n}r\right)\left(1-\sin^{n}r\right).
\end{align*}

Therefore, $\|\omega_N\|_{L^2} = 1$ when
\begin{equation}\label{eqn:g2rncn}
	C_N = \sqrt{\frac{(n-k)!}{\text{vol}\,V_2 \mathbb{R}^{n+1}\, (n-1)!\, k!\left(1+\frac{k}{n-k}\sin^{n}r\right)\left(1-\sin^{n}r\right)}}.
\end{equation}
The equivalent computation in the Dirichlet case gives that
\begin{equation}\label{eqn:g2rncd}
	C_D = \sqrt{\frac{(n-k)!}{\text{vol}\,V_2 \mathbb{R}^{n+1}\, (n-1)!\,(k-1)! \left(\frac{k}{n-k} + \sin^{n}r\right)\left(1-\sin^{n}r\right)}}.
\end{equation}

\subsection{Computing the Poincar\'e duality angle} 
\label{sec:g2rncomputing_the_poincar'e_duality_angle}
The Poincar\'e duality angle $\theta_r^{2k}$ is simply given by
\[
	\cos \theta_r^{2k} = \langle \omega_N, \omega_D\rangle_{L^2}
\]
since $\|\omega_N\|_{L^2} = \|\omega_D\|_{L^2} = 1$.

Using the definitions of $\omega_N$ and $\omega_D$, 
\begin{align*}
	\omega_N \wedge \star\, \omega_D & = \frac{(k-1)!}{(n-k-1)!} \left[ k \sin t \cos^{n-2k-1}t\, f_N(t) f_D(t) \vphantom{\frac{\cos^{n-2k+1}t}{\sin t}}\right.\\
	& \qquad \qquad \qquad \quad \left.+ \frac{1}{n-k} \frac{\cos^{n-2k+1}t}{\sin t} g_N(t)g_D(t)\right] \alpha \wedge (d\alpha)^{n-1} \wedge \tau,
\end{align*}
so
\begin{multline*}
	\cos \theta_r^{2k}  = \int_{N_r} \omega_N \wedge \star\, \omega_D \\
	 = Q \int_0^{\pi/2-r} \left[ k\sin t \cos^{n-2k-1}t\, f_N(t) f_D(t) + \frac{1}{n-k} \frac{\cos^{n-2k+1}t}{\sin t} g_N(t) g_D(t) \right] dt
\end{multline*}
after integrating out $\left(V_2 \mathbb{R}^{n+1}\right)_t$.

Using the expressions \eqref{eqn:g2rnfngn} and \eqref{eqn:g2rnfdgd} for $f_N(t)$ and $g_N(t)$, this reduces to
\begin{align}
	\nonumber \cos \theta_r^{2k} & = C_NC_DQ\int_0^{\pi/2-r}\left[\frac{kn}{n-k}\sin t \cos^{n-1}t - \frac{kn}{n-k} \sin^{2n}r\, \frac{\sin t}{\cos^{n+1}t} \right] dt \\
	\label{eqn:g2rntheta1} & = C_NC_DQ \frac{k}{n-k}(1-\sin^nr)^2.
\end{align}
The definitions for $Q$, $C_N$ and $C_D$ given in \eqref{eqn:g2rnq}, \eqref{eqn:g2rncn} and \eqref{eqn:g2rncd} yield
\begin{align*}
	C_NC_DQ \frac{k}{n-k} & = \frac{\sqrt{\frac{k}{n-k}}}{\left(1-\sin^nr\right) \sqrt{\left(1+\frac{k}{n-k}\sin^nr\right)\left(\frac{k}{n-k} + \sin^nr\right)}} \\
	& = \frac{1}{\left(1-\sin^nr\right)\sqrt{\left(1+\sin^nr\right)^2 + \frac{(n-2k)^2}{k(n-k)}\sin^n r}}.
\end{align*}

Plugging this in to \eqref{eqn:g2rntheta1} then gives that
\begin{equation}\label{eqn:g2rntheta}
	\cos \theta_r^{2k} = \frac{1-\sin^nr}{\sqrt{\left(1+\sin^nr\right)^2 + \frac{(n-2k)^2}{k(n-k)}\sin^nr}},
\end{equation}
completing the proof of Theorem~\ref{thm:g2rn}.


\subsection{Other Grassmannians} 
\label{sec:other_grassmannians}
The parallels between the above calculation of the Poincar\'e duality angles for $N_r$ to the calculation in the $\CP^n$ story suggests that a similar strategy may work for other Grassmannians.  

If $G_k \mathbb{R}^{n+1}$ is the Grassmannian of oriented $k$-planes in $\mathbb{R}^{n+1}$, then there are two obvious subGrassmannians,
\[
	G_k \mathbb{R}^{n} \quad \text{and} \quad G_{k-1} \mathbb{R}^{n},
\]
where a $(k-1)$-plane in $\mathbb{R}^{n}$ specifies the $k$-plane in $\mathbb{R}^{n+1}$ which contains it and the $x_{n+1}$-axis.  These subGrassmannians sit at distance $\pi/2$ from each other in $G_k \mathbb{R}^{n+1}$ and each is the other's focal locus.  Using the same ideas as in Section~\ref{sec:g2rngeometry}, $G_k \mathbb{R}^{n+1}$ is the cohomogeneity one manifold corresponding to the group diagram
\begin{equation*}
	\begin{diagram}[size=2em]
		& &  & SO(n) & &&  \\
		&& \ldLine & &   \rdLine &&\\
		SO(k) \times & SO(n-k) & & & & SO(k-1)&  \times SO(n-k+1) \\
		& &\rdLine & & \ldLine &&  \\
		 &  && SO(k-1) \times SO(n-k) &&  &  \\
	\end{diagram}
\end{equation*}

A principal orbit
\[
	\frac{SO(n)}{SO(k-1) \times SO(n-k)}
\]
forms a bundle over $G_k \mathbb{R}^n = \frac{SO(n)}{SO(k)\times SO(n-k)}$ with fiber
\[
	\frac{SO(k) \times SO(n-k)}{SO(k-1) \times SO(n-k)} = \frac{SO(k)}{SO(k-1)} = S^{k-1}.
\]
Each principal orbit is also a bundle over $G_{k-1} \mathbb{R}^{n} = \frac{SO(n)}{SO(k-1) \times SO(n-k+1)}$ with fiber
\[
	\frac{SO(k-1) \times SO(n-k+1)}{SO(k-1) \times SO(n-k)} = \frac{SO(n-k+1)}{SO(n-k)} = S^{n-k}.
\]
The fiber over $G_k \mathbb{R}^n$ is scaled by $\sin t$ and the fiber over $G_{k-1} \mathbb{R}^{n}$ is scaled by $\cos t$.

This then gives a fairly complete geometric picture of the manifold with boundary
\[
	N_r := G_k \mathbb{R}^{n+1} - \nu_r\left(G_{k-1}\mathbb{R}^{n}\right).
\]

Emulating the case when $k=2$, the prototype for a closed and co-closed $pk$-form on $N_r$ would then be 
\[
	\omega = f(t) (d\alpha)^p + g(t)\, \alpha \wedge (d\alpha)^{p-1} \wedge \tau
\]
where  $\alpha$ restricts to an $SO(n)$-invariant $(k-1)$-form on each principal orbit.

However, there may not always be Poincar\'e duality angles in the expected dimensions.  For example, when $k=3$ and $n=5$, the manifold
\[
	N_r := G_3 \mathbb{R}^6 - \nu_r\left(G_2 \mathbb{R}^5\right)
\]
has {\it no} Poincar\'e duality angles.  

To see this, note that $G_3 \mathbb{R}^6$ is a $3(6-3) = 9$-dimensional manifold, so $N_r$ is also a $9$-manifold.  Since $N_r$ is a $D^3$-bundle over the six-dimensional submanifold $G_3 \mathbb{R}^5 \simeq G_2 \mathbb{R}^5$, it has the same absolute cohomology as $G_2 \mathbb{R}^5$.  Hence, the cohomology of $N_r$ occurs in dimensions 0, 2, 4 and 6 since the closed manifold $G_2 \mathbb{R}^5$ has 1-dimensional real cohomology groups in dimensions 0, 2, 4 and 6.  Using Poincar\'e--Lefschetz duality, this means that $N_r$ has one-dimensional relative cohomology groups in dimensions 3, 5, 7 and 9.  Therefore, 
\[
	H^i(N_r; \mathbb{R}) \quad \text{and} \quad H^i(N_r, \partial N_r; \mathbb{R})
\]
cannot both be non-zero for any $i$, so there are no Poincar\'e duality angles.

Since $G_3 \mathbb{R}^5$ is homeomorphic to $G_2 \mathbb{R}^5$, the same holds even if $N_r$ is defined instead as
\[
	N_r := G_3 \mathbb{R}^6 - \nu_r\left(G_3 \mathbb{R}^5\right).
\]

This suggests that the next interesting case will be the Grassmannian $G_4 \mathbb{R}^8$, in which the subGrassmannians $G_3 \mathbb{R}^7$ and $G_4 \mathbb{R}^7$ are the focal loci of each other.  Since $G_3 \mathbb{R}^7$ and $G_4 \mathbb{R}^7$ are homeomorphic, the Poincar\'e duality angles should be the same regardless of which is removed to get a manifold with boundary.   It seems that finding harmonic representatives of the first Pontryagin form on $G_4 \mathbb{R}^8$ will be key in determining the Poincar\'e duality angles of the manifold $G_4 \mathbb{R}^8 - \nu_r\left(G_3 \mathbb{R}^7\right)$.  


\section{Connections with the Dirichlet-to-Neumann operator} 
\label{chap:dnmap}

The goal of this section is to elucidate the connection between the Poincar\'e duality angles and the Dirichlet-to-Neumann map for differential forms, then exploit that to find a partial reconstruction of the mixed cup product from boundary data.  Throughout this section, $M^n$ is a compact, oriented, smooth Riemannian manifold with non-empty boundary $\partial M$.

\subsection{The Dirichlet-to-Neumann map and Poincar\'e duality angles} 
\label{sec:dnpdangles}

Suppose $\theta_1^p, \ldots , \theta_\ell^p$ are the Poincar\'e duality angles of $M$ in dimension $p$; i.e. $\theta_1^p, \ldots , \theta_\ell^p$ are the principal angles between the interior subspaces $\mathcal{E}_\partial \hp_N(M)$ and $c\mathcal{E}\hp_D(M)$.  If $\mathsf{proj}_D: \hp_N(M) \to \hp_D(M)$ is the orthogonal projection onto the space of Dirichlet fields and $\mathsf{proj}_N: \hp_D(M) \to \hp_N(M)$ is the orthogonal projection onto the space of Neumann fields, then recall from Section~\ref{sub:poincar'e_duality_angles} that the $\cos^2 \theta_i^p$ are the non-zero eigenvalues of the composition
\[
	\mathsf{proj}_N \circ \mathsf{proj}_D: \hp_N(M) \to \hp_N(M).
\]

It is this interpretation of the Poincar\'e duality angles which yields the connection to the Dirichlet-to-Neumann map.  

Specifically, the Hilbert transform is closely related to the projections $\mathsf{proj}_D$ and $\mathsf{proj}_N$, as illustrated by the following propositions:

\begin{proposition}\label{prop:tproj}
	If $\omega \in \hp_N(M)$ and $\mathsf{proj}_D\, \omega = \eta \in \hp_D(M)$ is the orthogonal projection of $\omega$ onto $\hp_D(M)$, then
	\[
		Ti^*\omega = (-1)^{np+1}\,i^*\!\star \eta.
	\]
\end{proposition}

\begin{proposition}\label{prop:tprojdirichlet}
	If $\lambda \in \mathcal{H}^{p}_D(M)$ and $\mathsf{proj}_N\,\lambda = \sigma \in \hp_N(M)$ is the orthogonal projection of $\lambda$ onto $\mathcal{H}^{p}_N(M)$, then
	\[
		Ti^*\!\star \lambda = (-1)^{n+p+1}\,i^*\sigma.
	\]
\end{proposition}
Proposition~\ref{prop:tprojdirichlet} is proved by applying the Hodge star and then invoking Proposition~\ref{prop:tproj}.

\begin{proof}[Proof of Proposition~\ref{prop:tproj}]
	Using the first Friedrichs decomposition \eqref{eqn:friedrichs},
	\begin{equation}\label{eqn:projomegadecomp}
		\omega = \delta \xi + \eta \in c\mathcal{E}\hp(M) \oplus \hp_D(M).
	\end{equation}
	Since $\omega$ satisfies the Neumann boundary condition, 
	\[
		0 = i^*\!\star \omega = i^*\!\star(\delta \xi + \eta) = i^*\!\star\delta \xi + i^*\!\star \eta, 
	\]
	so 
	\begin{equation}\label{eqn:tprojeta}
		i^*\!\star \eta = - i^*\!\star \delta \xi.
	\end{equation}
	On the other hand, since $\eta$ satisfies the Dirichlet boundary condition,
	\[
		i^*\omega = i^*(\delta \xi + \eta) = i^*\delta \xi + i^*\eta = i^*\delta \xi = (-1)^{np+n+2}\, i^*\!\star d \star \xi.
	\]
	The form $\xi$ can be chosen such that 
	\[
		\Delta \xi = 0 \quad \text{and} \quad d \xi = 0 
	\]
	(cf.\ \cite[(4.11)]{Schwarz} or \cite[Section 2]{Belishev}), which means that $\star\, \xi$ solves the boundary value problem
	\[
		\Delta \varepsilon = 0, \quad i^*\varepsilon = i^*\!\star \xi, \quad \text{and} \quad i^*\delta \varepsilon  = 0.
	\]
	Therefore, 
	\[
		(-1)^{np+n+2}\,\Lambda i^*\!\star \xi = (-1)^{np+n+2}\,i^*\!\star d \star \xi = i^*\omega.
	\]
	Then, using the definition of the Hilbert transform $T = d_\partial \Lambda^{-1}$ and  \eqref{eqn:tprojeta},
	\begin{align*}
		Ti^*\omega = (-1)^{np+n+2}\, d_\partial\Lambda^{-1} \Lambda i^*\!\star \xi & = (-1)^{np+n}\, d_\partial i^*\!\star \xi \\ 
		& = (-1)^{np+2}\, i^*\!\star \delta \xi \\
		& = (-1)^{np+1}\, i^*\!\star \eta,
	\end{align*}
	as desired.
\end{proof}

Taken together, Propositions~\ref{prop:tproj} and \ref{prop:tprojdirichlet} can be seen as a clarification of Belishev and Sharafutdinov's Lemma 7.1, which says that $T$ maps $i^*\hp_N(M)$ to $i^*\mathcal{H}^{n-p}_N(M)$.

Consider the restriction $\til{T}$ of the Hilbert transform $T$ to $i^*\hp_N(M)$.  Since $\til{T}$ is closely related to the orthogonal projections $\mathsf{proj}_D$ and $\mathsf{proj}_N$, it should come as no surprise that $\til{T}^2$ is closely related to the composition $\mathsf{proj}_N \circ \mathsf{proj}_D$, the eigenvalues of which are the $\cos^2 \theta_i^p$.  Indeed, the connection between the Dirichlet-to-Neumann map and the Poincar\'e duality angles is given by:

\begin{theoremeigenvalues}
	If $\theta_1^p, \ldots , \theta_\ell^p$ are the principal angles between $\mathcal{E}_\partial\hp_N(M)$ and $c\mathcal{E}_\partial\hp_D(M)$ (i.e. the Poincar\'e duality angles in dimension $p$), then the quantities
	\[
		(-1)^{pn+p+n}\cos^2\theta_i^p
	\]
	are the non-zero eigenvalues of $\til{T}^2$. 
\end{theoremeigenvalues}
\begin{proof}
	If $\delta \alpha \in c\mathcal{E}\hp_N(M)$, the boundary subspace of $\hp_N(M)$, then, by Theorem~\ref{thm:dg2}, the form $\delta \alpha$ is orthogonal to $\hp_D(M)$.  By \propref{prop:tproj}, then, $\til{T}i^*\delta \alpha = 0$,  so $i^*c\mathcal{E}\hp_N(M)$ is contained in the kernel of $\til{T}^2$.
	
	Combining Propositions~\ref{prop:tproj} and \ref{prop:tprojdirichlet}, the eigenforms of $\til{T}^2$ are precisely the eigenforms of $\mathsf{proj}_N \circ \mathsf{proj}_D$.  Moreover, if $\omega_i \in \mathcal{E}_\partial \hp_N(M)$ is an eigenform of $\mathsf{proj}_N \circ \mathsf{proj}_D$ corresponding to a non-zero eigenvalue, then $\omega_i$ is a Neumann field realizing the Poincar\'e duality angle $\theta_i^p$.  Hence
	\[
		\mathsf{proj}_N \circ \mathsf{proj}_D(\omega_i) = \cos^2 \theta_i^p\, \omega_i.
	\]
	Therefore,
	\[
		\til{T}^2i^*\omega_i = T^2i^*\omega_i = (-1)^{pn+p+n} \cos^2 \theta_i^p\, i^* \omega_i,
	\]
	so the $(-1)^{pn+p+n}\cos^2\theta_i^p$ are the non-zero eigenvalues of $\til{T}^2$.
\end{proof}

Note that the domain of $\til{T}^2$ is $ i^*\hp_N(M) = \text{im}\, G_{n-p-1}$, which is determined by the Dirichlet-to-Neumann map.  Thus, \thmref{thm:pdangleseigenvalues} implies that the Dirichlet-to-Neumann operator not only determines the cohomology groups of $M$, as shown by Belishev and Sharafutdinov, but determines the interior and boundary cohomology.  

\begin{corollary}\label{cor:dninteriorboundary}
	The boundary data $(\partial M, \Lambda)$ distinguishes the interior and boundary cohomology of $M$.
\end{corollary}

\begin{proof}
	By \thmref{thm:pdangleseigenvalues}, the pullback $i^*c\mathcal{E}\hp_N(M)$ of the boundary subspace is precisely kernel of the operator $\til{T}^2$, while the pullback $i^*\mathcal{E}_\partial \hp_N(M)$ of the interior subspace is the image of $\til{T}^2$.  Since harmonic Neumann fields are uniquely determined by their pullbacks to the boundary, 
\[
	i^*c\mathcal{E}\hp_N(M) \cong c\mathcal{E}\hp_N(M) \quad \text{ and } \quad i^*\mathcal{E}_\partial \hp_N(M) \cong \mathcal{E}_\partial \hp_N(M),
\]
so the interior and boundary absolute cohomology groups are determined by the data $(\partial M, \Lambda)$.  Since, for each $p$, $\star\, c\mathcal{E}\hp_N(M) = \mathcal{EH}^{n-p}_D(M)$ and $\star\, \mathcal{E}_\partial \hp_N(M) = c\mathcal{E}_\partial \mathcal{H}^{n-p}_D(M)$, the interior and boundary relative cohomology groups are also determined by the data $(\partial M, \Lambda)$. 
\end{proof}


\subsection{A decomposition of the traces of harmonic fields} 
\label{sec:a_decomposition_of_the_traces_of_harmonic_fields}

Aside from the obvious connection that it gives between the Dirichlet-to-Neumann map and the Poincar\'e duality angles, Theorem~\ref{thm:pdangleseigenvalues} also implies that the traces of harmonic Neumann fields have the following direct-sum decomposition:
\begin{equation}\label{eqn:hpndecomp}
	i^*\hp_N(M) = i^*c\mathcal{E}\hp_N(M) + i^*\mathcal{E}_\partial \hp_N(M).
\end{equation}
Since $\hp_N(M) = c\mathcal{E}\hp_N(M) \oplus \mathcal{E}_\partial \hp_N(M)$, the space $i^*\hp_N(M)$ is certainly the sum of the subspaces $i^*c\mathcal{E}\hp_N(M)$ and $i^*\mathcal{E}_\partial \hp_N(M)$.  \thmref{thm:pdangleseigenvalues} implies that $i^*c\mathcal{E}\hp_N(M)$ is contained in the kernel of $\til{T}^2$ and that $\til{T}^2$ is injective on $i^*\mathcal{E}_\partial \hp_N(M)$.  Therefore, the spaces $i^*c\mathcal{E}\hp_N(M)$ and $i^*\mathcal{E}_\partial \hp_N(M)$ cannot overlap, so the sum in \eqref{eqn:hpndecomp} is indeed direct.

In fact, slightly more is true.  Removing the restriction on the domain, the operator $T^2$ is a map $i^*\hp(M) \to i^*\hp(M)$.  To see this, recall that $T = d_\partial \Lambda^{-1}$, so the domain of $T^2$ is certainly
\[
	\text{im } \Lambda_{n-p-1} = i^*\hp(M).
\]
Moreover (again using the fact that $T = d_\partial \Lambda^{-1}$), the image of $T^2$ is contained in $\mathcal{E}^p(\partial M)$ which, by \thmref{thm:bs1}, is contained in $\ker \Lambda_p = i^*\hp(M)$.  

By \thmref{thm:dg3}, the space $\hp(M)$ of harmonic $p$-fields on $M$ admits the decomposition
\begin{equation}\label{eqn:hpmdg3decomp}
	\hp(M) = \mathcal{E}c\mathcal{E}^p(M) \oplus \left( \hp_N(M) + \hp_D(M)\right).
\end{equation}
Since $\hp_N(M)$ can be further decomposed as 
\[
	\hp_N(M) = c\mathcal{E}\hp_N(M) \oplus \mathcal{E}_\partial \hp_N(M)
\]
and since $i^*\hp_D(M) = \{0\}$, this means that
\begin{equation}\label{eqn:istarhpmdecomp}
	i^*\hp(M) = i^*\mathcal{E}c\mathcal{E}^p(M) + i^*c\mathcal{E}\hp_N(M) + i^*\mathcal{E}_\partial \hp_N(M),
\end{equation}
where the right hand side is just a sum of spaces and not, {\it a priori}, a direct sum. However, two harmonic fields cannot pull back to the same form on the boundary unless they differ by a Dirichlet field.  Since \eqref{eqn:hpmdg3decomp} is a direct sum, elements of $\mathcal{E}c\mathcal{E}^p(M) \oplus c\mathcal{E}\hp_N(M) \oplus \mathcal{E}_\partial \hp_N(M)$ cannot differ by a Dirichlet field, so the right hand side of \eqref{eqn:istarhpmdecomp} is a direct sum. A characterization of the summands follows from Theorem~\ref{thm:pdangleseigenvalues} and a description of how $T^2$ acts on $i^*\mathcal{E}c\mathcal{E}^p(M)$.

\begin{lemma}\label{lem:t2onecep}
	The restriction of $T^2$ to the subspace $i^*\mathcal{E}c\mathcal{E}^p(M)$ is $(-1)^{np+p}$ times the identity map.
\end{lemma}

\begin{proof}
	Suppose $\varphi \in i^*\mathcal{E}c\mathcal{E}^p(M)$.  Then 
	\[
		\varphi = i^*d\gamma  = i^*\delta \xi
	\] 
	for some form $d\gamma = \delta \xi \in \mathcal{E}c\mathcal{E}^p(M)$.  The form $\xi \in \Omega^{p+1}(M)$ can be chosen such that
	\[
		\Delta \xi = 0, \quad d\xi = 0.
	\]
	Therefore, 
	\[
		\Delta \star \xi = \star\, \Delta \xi = 0, \quad i^*\delta \star \xi = (-1)^p\, i^*\!\star d \xi = 0,
	\]
	so
	\[
		\varphi = i^*\delta \xi = (-1)^{np+1}\, i^*\!\star d \star \xi = (-1)^{np+1}\, \Lambda i^*\!\star \xi.
	\]
	Since $\gamma \in \Omega^{p-1}(M)$ can be chosen such that
	\[
		\Delta \gamma = 0, \quad \delta \gamma = 0, 
	\]
	this means that 
	\[
		\Lambda i^*\gamma = i^*\!\star d\gamma = i^*\!\star \delta \xi = (-1)^{p+1}\, i^*d\star \xi = (-1)^{p+1} d_\partial i^*\!\star \xi = (-1)^{np+p} d_\partial \Lambda^{-1} \varphi.
	\]
	
	Hence, applying $T^2$ to $\varphi$ yields
	\[
		T^2 \varphi = T d_\partial \Lambda^{-1} \varphi = (-1)^{np+p}\, T \Lambda i^*\gamma = (-1)^{np+p} d_\partial \Lambda^{-1} \Lambda i^*\gamma.
	\]
	Simplifying further, this implies that
	\[
		T^2 \varphi = (-1)^{np+p} d_\partial i^*\gamma = (-1)^{np+p}\, i^*d\gamma = (-1)^{np+p}\, \varphi.
	\]
	Since the choice of $\varphi \in i^*\mathcal{E}c\mathcal{E}^p(M)$ was arbitrary, this implies that $T^2$ is $(-1)^{np+p}$ times the identity map on $i^*\mathcal{E}c\mathcal{E}^p(M)$, completing the proof of the lemma.
\end{proof}

Lemma~\ref{lem:t2onecep} has the following immediate consequence:

\begin{proposition}\label{prop:hpdecomp}
	The data $(\partial M, \Lambda)$ determines the direct-sum decomposition 
	\begin{equation}\label{eqn:hpdecomp}
		\ker \Lambda_p = i^*\hp(M) = i^*\mathcal{E}c\mathcal{E}^p(M) + i^*c\mathcal{E}\hp_N(M) + \mathcal{E}_\partial \hp_N(M),
	\end{equation}
\end{proposition}

\begin{proof}
Since the sum in \eqref{eqn:hpdecomp} was already seen to be direct, Lemma~\ref{lem:t2onecep} implies that $i^*\mathcal{E}c\mathcal{E}^p(M)$ is the $(-1)^{np+p}$-eigenspace of $T^2$.  Likewise, Theorem~\ref{thm:pdangleseigenvalues} says that $i^*c\mathcal{E}\hp_N(M)$ is the kernel of $T^2$ and $i^*\mathcal{E}_\partial \hp_N(M)$ is the sum of the eigenspaces corresponding to the eigenvalues $\cos^2 \theta_i^p$.  Since $T^2$ is certainly determined by $(\partial M , \Lambda)$, this completes the proof of the proposition.
\end{proof}

The decomposition \eqref{eqn:hpdecomp} turns out not, in general, to be orthogonal; equivalently, the operator $T^2$ is not self-adjoint.  In particular, as will be shown in \thmref{thm:kernel}, 
\[
	i^*\mathcal{E}c\mathcal{E}^p(M) + i^*\mathcal{E}_\partial \hp_N(M) = \mathcal{E}^p(\partial M),
\]
but $i^*c\mathcal{E}\hp_N(M)$ is not usually contained in $\hp(\partial M)$.  This is somewhat surprising since elements of $c\mathcal{E}\hp_N(M)$ are harmonic fields and their pullbacks to $\partial M$ must be non-trivial in cohomology.

However, the decomposition \eqref{eqn:hpdecomp} does yield a refinement of \thmref{thm:bs1}:

\begin{kertheorem}
	Let $\Lambda_p:\Omega^p(\partial M)\to \Omega^{n-p-1}(\partial M)$ be the Dirichlet-to-Neumann operator.  Then $\ker \Lambda_p$ has the direct-sum decomposition
	\begin{equation}\label{eqn:kerlambda}
		\ker \Lambda_p = i^* \hp(M) = i^*c\mathcal{E}\hp_N(M) + \mathcal{E}^p(\partial M).
	\end{equation}
	Therefore, 
	\[
		 \ker \Lambda_p/\mathcal{E}^p(\partial M) \cong c\mathcal{E}\hp_N(M),
	\]
	so the dimension of this space is equal to the dimension of the boundary subspace of $H^p(M; \mathbb{R})$.
\end{kertheorem}

\begin{proof}
	Using \eqref{eqn:hpdecomp}, the decomposition \eqref{eqn:kerlambda} will follow from the fact that
	\begin{equation}\label{eqn:epdecomp}
		\mathcal{E}^p(\partial M) = i^*\mathcal{E}c\mathcal{E}^p(M) +i^*\mathcal{E}_\partial \hp_N(M).
	\end{equation}
	The right hand side is certainly contained in the space on the left.  To see the other containment, suppose $d_\partial\varphi \in \mathcal{E}^p(\partial M)$.  Then let $\varepsilon \in \Omega^{p-1}(M)$ solve the boundary value problem
	\[
		\Delta \varepsilon = 0, \quad i^*\varepsilon = \varphi, \quad i^*\delta \varepsilon = 0.
	\]
	Then $d\varepsilon \in \hp(M)$ and $i^*d\varepsilon = d_\partial i^*\varepsilon = d_\partial\varphi$.  Now, using \thmref{thm:dg3},
	\[
		d\varepsilon = d\gamma + \lambda_N + \lambda_D \in \mathcal{E}c\mathcal{E}^p(M) \oplus \left( \hp_N(M) + \hp_D(M)\right).
	\]
	Since $\lambda_D$ is a Dirichlet field,
	\[
		d_\partial\varphi = i^*d\varepsilon = i^*d\gamma + i^*\lambda_N.
	\]
	Therefore, 
	\[
		i^*\lambda_N = d_\partial \varphi - d_\partial i^*\gamma = d_\partial (\varphi - i^*\gamma)
	\]
	is an exact form, so $\lambda_N \in \mathcal{E}_\partial \hp_N(M)$.  Hence,
	\[
		d_\partial \varphi = i^*d\gamma + i^*\lambda_N \in i^*\mathcal{E}c\mathcal{E}^p(M) + i^*\mathcal{E}_\partial \hp_N(M),
	\]
	so $\mathcal{E}^p(\partial M) \subset i^*\mathcal{E}c\mathcal{E}^p(M) + i^*\mathcal{E}_\partial \hp_N(M)$ and \eqref{eqn:epdecomp} follows.
	
	The decompositions \eqref{eqn:hpdecomp} and \eqref{eqn:epdecomp} together imply that
	\[
		\ker \Lambda_p = i^*\hp(M) = i^*c\mathcal{E}\hp_N(M) + \mathcal{E}^p(\partial M),
	\]
	meaning that
	\[
		 \ker \Lambda_p / \mathcal{E}^p(\partial M)  \cong i^*c\mathcal{E}\hp_N(M).
	\]
	However, since the elements of $c\mathcal{E}\hp_N(M)$ are harmonic Neumann fields and harmonic Neumann fields are uniquely determined by their pullbacks to the boundary, $i^*c\mathcal{E}\hp_N(M) \cong c\mathcal{E}\hp_N(M)$, so
	\[
		\ker \Lambda_p / \mathcal{E}^p(\partial M)  \cong c\mathcal{E}\hp_N(M),
	\]
	as desired.
\end{proof}


\subsection{Partial reconstruction of the mixed cup product} 
\label{sec:mixed}
The goal of this section and the next is to state and prove Theorem~\ref{thm:mixedcupproductboundary}, which says that the mixed cup product 
\[
	\cup: H^p(M; \mathbb{R}) \times H^q(M, \partial M; \mathbb{R}) \to H^{p+q}(M, \partial M; \mathbb{R})
\]
can be at least partially recovered from the boundary data $(\partial M, \Lambda)$ for any $p$ and $q$. 

Since $H^p(M; \mathbb{R}) \cong \text{im } G_{n-p-1}$ and
\[
	H^q(M, \partial M; \mathbb{R}) \cong H^{n-q}(M; \mathbb{R}) \cong \text{im } G_{q-1},
\]
an absolute cohomology class $[\alpha] \in H^p(M; \mathbb{R})$ and a relative cohomology class $[\beta] \in H^q(M, \partial M; \mathbb{R})$  correspond, respectively, to forms on the boundary
\[
	\varphi \in \text{im }G_{n-p-1} \subset \Omega^p(\partial M) \quad \text{and} \quad \psi \in \text{im } G_{q-1} \subset \Omega^{n-q}(\partial M).
\]
In turn, the relative cohomology class $$[\alpha] \cup [\beta] \in H^{p+q}(M, \partial M; \mathbb{R}) \cong H^{n-p-q}(M; \mathbb{R})$$ corresponds to a form 
\[
	\theta \in \text{im } G_{p+q-1} \subset \Omega^{n-p-q}(\partial M).
\]

More concretely, the class $[\alpha]$ is represented by the Neumann harmonic field $\alpha \in \hp_N(M)$, the class $[\beta]$ is represented by the Dirichlet harmonic field $\beta \in \hq_D(M)$ and
\[
	\varphi = i^*\alpha, \qquad \psi = i^*\!\star \beta.
\]
The form $\alpha \wedge \beta \in \Omega^{p+q}(M)$ is closed since
\[
	d(\alpha \wedge \beta) = d\alpha \wedge \beta +(-1)^p\, \alpha \wedge d\beta = 0.
\]
Therefore, using the second Friedrichs decomposition \eqref{eqn:friedrichs2},
\begin{equation}\label{eqn:mixedalphawedgebeta}
	\alpha \wedge \beta = \delta \xi + \eta + d\zeta \in c\mathcal{EH}^{p+q}(M) \oplus \hpq_D(M) \oplus \mathcal{E}^{p+q}_D(M)
\end{equation}
The form $\alpha \wedge \beta$ satisfies the Dirichlet boundary condition (see Section~\ref{sub:_alpha_wedge_beta_is_a_dirichlet_form} for details), so the relative cohomology class $[\alpha] \cup [\beta] = [\alpha \wedge \beta]$ is represented by the form $\eta \in \hpq_D(M)$ and 
\[
	\theta = i^*\!\star\eta.
\]
Reconstructing the mixed cup product would be equivalent to showing that the form $i^*\!\star\eta$ can be determined from the forms $\varphi = i^*\alpha$ and $\psi = i^*\!\star\beta$.  This can be done in the case that $\beta$ comes from the boundary subspace of $\hq_D(M)$:

\begin{theoremmixedboundary}
	The boundary data $(\partial M, \Lambda)$ completely determines the mixed cup product when the relative cohomology class comes from the boundary subspace.  More precisely, with notation as above, if $\beta \in \mathcal{EH}^q_D(M)$, then 
	\[
		i^*\!\star \eta = (-1)^p\,\Lambda(\varphi \wedge \Lambda^{-1}\psi).
	\] 
\end{theoremmixedboundary}

When $M$ can be embedded in the Euclidean space $\mathbb{R}^n$ of the same dimension, all of the cohomology of $M$ must be carried by the boundary.  Thus, the following, which may be relevant to Electrical Impedance Tomography, is an immediate corollary of \thmref{thm:mixedcupproductboundary}:

\begin{corollaryeuclidean}
	If $M^n$ is a compact region in $\mathbb{R}^n$, the boundary data $(\partial M, \Lambda)$ completely determines the mixed cup product on $M$.
\end{corollaryeuclidean}

The obvious conjecture is that the result of \thmref{thm:mixedcupproductboundary} holds without the hypothesis that $\beta$ comes from the boundary subspace:

\begin{conjecture}\label{conj:mixedcupproduct}
	The boundary data $(\partial M, \Lambda)$ determines the mixed cup product on $M$.  In particular, with notation as above,
	\[
		i^*\!\star \eta = (-1)^p\,\Lambda(\varphi \wedge \Lambda^{-1} \psi).
	\]
\end{conjecture}


\subsection{The proof of Theorem~\ref{thm:mixedcupproductboundary}} 
\label{sec:proof_mixed_cup_product}

\subsubsection{$\Lambda(\varphi \wedge \Lambda^{-1}\psi)$ is well-defined} 
\label{sub:mixed:lambdawelldefined}

Since $\star\, \beta$ is a harmonic field,
\[
	\psi = i^*\!\star\beta \in i^*\mathcal{H}^{n-q}(M) = \text{im } \Lambda_{q-1}
\]
by \lemref{lem:harmonictraces}.  Therefore, $\psi = \Lambda \mu$ for some $\mu \in \Omega^{q-1}(\partial M)$, so the expression $\Lambda^{-1}\psi = \mu$ seems to make sense.

Of course, $\Lambda$ has a large kernel, so the expression $\Lambda^{-1} \psi$ is not well-defined: for any $(q-1)$-form $\sigma \in \ker \Lambda$, the form $\mu + \sigma$ is another valid choice for $\Lambda^{-1}\psi$.  To see that this ambiguity does not matter, it suffices to show that 
\[
	\Lambda (\varphi \wedge (\mu + \sigma)) = \Lambda(\varphi \wedge \mu)
\]
for any such $\sigma \in \ker \Lambda$.
Certainly,
\begin{equation}\label{eqn:mixedlambdawelldefined}
	\Lambda (\varphi \wedge (\mu + \sigma)) = \Lambda(\varphi \wedge \mu) + \Lambda(\varphi \wedge \sigma),
\end{equation}
so the goal is to show that $\Lambda(\varphi \wedge \sigma) = 0$.  

Since the kernel of $\Lambda$ consists of pullbacks of harmonic fields, there is some $\tau \in \mathcal{H}^{q-1}(M)$ such that $\sigma = i^*\tau$.  Then
\[
	\varphi \wedge \sigma = i^*\alpha \wedge i^*\tau = i^*(\alpha \wedge \tau).
\]
Both $\alpha$ and $\tau $ are harmonic fields, so the form $\alpha \wedge \tau$ is closed, meaning that
\[
	\alpha \wedge \tau = \chi + d\varepsilon \in \mathcal{H}^{p+q-1}(M) \oplus \mathcal{E}^{p+q-1}_D(M)
\]
by the Morrey decomposition \eqref{eqn:morrey}.  Since $i^*d\varepsilon = 0$, 
\[
	\varphi \wedge \sigma = i^*(\alpha \wedge \tau) = i^*\chi,
\]
so $\chi$ solves the boundary value problem
\[
	\Delta \chi = 0, \quad i^*\chi = \varphi\wedge\sigma, \quad i^*\delta \chi = 0.
\]

Thus, by definition of the Dirichlet-to-Neumann map,
\[
	\Lambda(\varphi \wedge \sigma) = i^*\!\star d\chi = 0
\]
since $\chi$ is closed.
Applying this to \eqref{eqn:mixedlambdawelldefined} gives that
\[
	\Lambda(\varphi \wedge (\mu + \sigma)) = \Lambda(\varphi \wedge \mu),
\]
so the expression $\Lambda(\varphi \wedge \Lambda^{-1}\psi)$ is indeed well-defined.

The above argument only depended on the fact that  $\beta \in \hq_D(M)$, so the expression $\Lambda(\varphi \wedge \Lambda^{-1} \psi)$ is well-defined regardless of whether or not $\beta$ lives in the boundary subspace.  Thus, Conjecture~\ref{conj:mixedcupproduct} is at least plausible.


\subsubsection{$\alpha \wedge \beta$ is a Dirichlet form} 
\label{sub:_alpha_wedge_beta_is_a_dirichlet_form}
Since $\beta$ satisfies the Dirichlet boundary condition, $i^*\beta = 0$ and
\[
	i^*(\alpha \wedge \beta) = i^*\alpha \wedge i^*\beta = 0.
\]
In other words, $\alpha \wedge \beta$ satisfies the Dirichlet boundary condition.  Using the decomposition \eqref{eqn:mixedalphawedgebeta}, this means that
\[
	0 = i^*(\alpha \wedge \beta) = i^*(\delta \xi + \eta + d\gamma) = i^*\delta \xi,
\]
since $\eta$ and $d\gamma$ are both Dirichlet forms.  Thus, $\delta \xi$ satisfies the Dirichlet boundary condition.  However, the second Friedrichs decomposition \eqref{eqn:friedrichs2} says that $c\mathcal{EH}^{p+q}(M)$ is orthogonal to $\hpq_D(M)$, so $\delta \xi$ must be zero.

Hence, the decomposition \eqref{eqn:mixedalphawedgebeta} can be simplified as
\begin{equation}\label{eqn:mixed:decomp}
	\alpha \wedge \beta = \eta + d\gamma.
\end{equation}

\subsubsection{The proof of \thmref{thm:mixedcupproductboundary}} 
\label{sub:if_beta_is_a_boundary_form}

Suppose $\beta$ comes from the boundary subspace of $\mathcal{H}^q_D(M)$; i.e. 
\[
	\beta = d\beta_1 \in \mathcal{EH}^q_D(M).
\]
Since $\star\, d\beta_1$ is a harmonic field, \lemref{lem:harmonictraces} implies that $i^*\!\star d\beta_1$ is in the image of $\Lambda$.  In fact, since $\beta_1$ can be chosen to be harmonic and co-closed, $\beta_1$ solves the boundary value problem
\[
	\Delta \varepsilon = 0 , \quad i^*\varepsilon = i^*\beta_1, \quad i^*\delta \varepsilon = 0.
\]
Hence, 
\[
	\psi = i^*\!\star d\beta_1 = \Lambda i^*\beta_1.
\]
Therefore, $\Lambda^{-1}\psi = i^*\beta_1$ (up to the ambiguity mentioned in Section~\ref{sub:mixed:lambdawelldefined}), so
\[
	\Lambda(\varphi \wedge \Lambda^{-1} \psi) = \Lambda (\varphi \wedge i^*\beta_1).
\]

On the other hand, 
\[
	\alpha \wedge d\beta_1 = (-1)^p\, d(\alpha \wedge \beta_1)
\]
is exact, so $\eta$ is also exact:
\[
	\eta = \alpha \wedge d\beta_1 - d\zeta = d\left[(-1)^p\, \alpha \wedge \beta - \zeta \right].
\]
Letting $\eta' := (-1)^p\, \alpha \wedge \beta - \zeta$, the form $\eta = d\eta' \in \mathcal{EH}^{p+q}_D(M)$ belongs to the boundary subspace of $\hpq_D(M)$.  

Substituting into the decomposition \eqref{eqn:mixed:decomp} gives that
\[
	\alpha \wedge d\beta_1 = d\eta' + d\zeta,
\]
so the goal is to show that $(-1)^p\, \Lambda(\varphi \wedge \Lambda^{-1}\psi) = i^*\!\star\, d\eta'$.  Using the definition of $\varphi$ and the fact that $\Lambda^{-1}\psi = i^*\!\beta_1$, 
\begin{equation}\label{eqn:mixed:alphawedgebeta1}
	(-1)^p\, \varphi \wedge \Lambda^{-1}\psi = (-1)^p\, i^*\alpha \wedge i^*\beta_1 = (-1)^p\, i^*(\alpha \wedge \beta_1).
\end{equation}

Moreover, since $\alpha$ is closed,
\[
	d\left((-1)^p\, \alpha \wedge \beta_1\right) = \alpha \wedge d\beta_1 = d\eta' + d\zeta,
\]
so the conclusion that
\[
	(-1)^p\, \Lambda (\varphi \wedge \Lambda^{-1}\psi) = i^*\!\star d\eta'
\]
will follow from:

\begin{proposition}\label{prop:mixedprimitives}
	Let $m$ be an integer such that $1 \leq m \leq n$.  Given an exact Dirichlet form
	\[
		d\rho + d\varepsilon \in \mathcal{EH}_D^m(M) \oplus \mathcal{E}^m_D(M),
	\]
	suppose $\gamma \in \Omega^{m-1}(M)$ is any primitive of $d\rho + d\varepsilon$; i.e. $d\gamma = d\rho + d\varepsilon$.  Then
	\[
		\Lambda i^*\gamma = i^*\!\star d\rho.
	\]
\end{proposition}

To see that \thmref{thm:mixedcupproductboundary} follows, note that $(-1)^p\,\alpha \wedge \beta_1$ is a primitive for $\alpha \wedge d\beta_1 = d\eta' + d\zeta$, so Proposition~\ref{prop:mixedprimitives} and \eqref{eqn:mixed:alphawedgebeta1} imply that
\[
	i^*\!\star d\eta' = \Lambda i^*((-1)^p\, \alpha \wedge \beta_1) = (-1)^p\, \Lambda (\varphi \wedge \Lambda^{-1}\psi),
\]
completing the proof of \thmref{thm:mixedcupproductboundary}.

\begin{proof}[Proof of Proposition~\ref{prop:mixedprimitives}]
	First, note that a primitive $\rho$ for $d\rho$ can be chosen such that $\Delta \rho = 0$ and $\delta \rho = 0$.  Also, by definition of the space $\mathcal{E}^m_D(M)$, a primitive $\varepsilon$ for $d\varepsilon$ can be chosen such that $i^*\varepsilon = 0$.  Then $\rho + \varepsilon$ is a primitive for $d\rho + d\varepsilon$ and
	\[
		i^*(\rho + \varepsilon) = i^*\rho.
	\]
	Since $\rho$ is harmonic and co-closed, 
	\begin{equation}\label{eqn:mixed:etazeta}
		\Lambda\, i^*(\rho + \varepsilon) = \Lambda\, i^*\rho = i^*\!\star d\rho.
	\end{equation}
	
	Now, suppose $\gamma$ is another primitive of $d\rho + d\varepsilon$.  Then the form $ \gamma -\rho -\varepsilon $ is closed, and so can be decomposed as
	\[
		 \gamma - \rho - \varepsilon  = \kappa_1 + d\kappa_2 \in \mathcal{H}^{m-1}(M) \oplus \mathcal{E}^{m-1}_D(M).
	\]
	
	Then
	\[
		i^*( \gamma - \rho - \varepsilon) = i^*(\kappa_1 + d\kappa_2) = i^*\kappa_1
	\]
	since $d\kappa_2$ is a Dirichlet form.  Using \lemref{lem:harmonictraces}, this means that
	\[
		\Lambda i^*( \gamma -\rho - \varepsilon) = \Lambda i^* \kappa_1 = 0
	\]
	since $i^*\kappa_1 \in i^*\mathcal{H}^{m-1}(M)= \ker \Lambda$.  
	
	Combining this with \eqref{eqn:mixed:etazeta} gives that
	\[
		\Lambda i^* \gamma = \Lambda i^*(\gamma - \rho -\varepsilon + \rho + \varepsilon) = \Lambda i^*(\gamma - \rho - \varepsilon) + \Lambda i^*(\rho + \varepsilon) = i^*\!\star d\rho,
	\]
	as desired.
\end{proof}




\bibliographystyle{amsalpha}
\nocite{*}
\bibliography{poincare2}

\end{document}